\documentclass[11pt]{article}
\usepackage{mathrsfs}
\usepackage{amsfonts}
\usepackage{amssymb,amsmath,amsthm}
\usepackage{color}
\usepackage{verbatim} 

\newcommand{\R}{\mathbb{R}}

\newtheorem{theorem}{Theorem}[section]

\newtheorem{lemma}[theorem]{Lemma}
\newtheorem{proposition}{Proposition}[section]

\newtheorem{definition}[theorem]{Definition}
\newtheorem{remark}[theorem]{Remark}

\begin{document}
\title{\LARGE\bf{ Bifurcation into spectral gaps\\ for strongly indefinite Choquard equations} }
\date{}
\author{Huxiao Luo$^{1,2}$,  Bernhard Ruf$^{2}$, Cristina Tarsi$^{2,}$$\thanks{{\small E-mail: luohuxiao@zjnu.edu.cn (H. Luo), bernhard.ruf@unimi.it (B. Ruf), cristina.tarsi@unimi.it (C. Tarsi)}}$\\
{\small 1 Department of Mathematics, Zhejiang Normal University, Jinhua, Zhejiang, 321004, P. R. China } \\
{\small 2 Dipartimento di Matematica, Universit\`{a} degli Studi di Milano, Via C. Saldini, 50, 20133 Milan, Italy}
}
 \maketitle
\begin{center}
\begin{minipage}{13cm}
\par
\small  {\bf Abstract:} 
We consider the semilinear elliptic equations
\begin{equation*}
\left\{
\begin{array}{ll}
\aligned
&-\Delta u+V(x)u=\left(I_\alpha\ast |u|^p\right)|u|^{p-2}u+\lambda u\quad \text{for}~~x\in\R^N,  \\
&u(x) \to 0 ~~\text{as}~~ |x| \to\infty,
\endaligned
\end{array}
\right.
\end{equation*}
where $I_\alpha$ is a Riesz potential, $p\in(\frac{N+\alpha}N,\frac{N+\alpha}{N-2})$, $N\geq3$, and $V $ is continuous periodic. We assume that $0$ lies in the spectral gap $(a,b)$ of $-\Delta + V$.
We prove the existence of infinitely many geometrically distinct solutions in $H^1(\R^N)$ for each $\lambda\in(a, b)$, which bifurcate from $b$ if $\frac{N+\alpha}N< p < 1 +\frac{2+\alpha}{N}$. Moreover, $b$ is the unique gap-bifurcation  point (from zero) in $[a,b]$. When $\lambda=a$, we find infinitely many geometrically distinct solutions in $H^2_{loc}(\R^N)$. Final remarks are given  about the eventual occurrence of a bifurcation from infinity in $\lambda=a$.



\end{minipage}
\end{center}

 {\section{Introduction}}
 \setcounter{equation}{0}

The purpose of this paper is to study bifurcation into spectral gaps for a class of nonlinear and nonlocal Schr\"odinger equations with
periodic potential. More precisley, let us consider  the following equation
\begin{equation}\label{Choqeq}
\left\{
\begin{array}{ll}
\aligned
&-\Delta u+V(x)u=\left(I_\alpha\ast |u|^p\right)|u|^{p-2}u+\lambda u\quad \text{for}~~x\in\R^N, \\
&\qquad u(x) \to 0 ~~\text{as}~~ |x| \to\infty,
\endaligned
\end{array}
\right.
\end{equation}
where $N\geq3$, $p\in[2,\frac{N+\alpha}{N-2})$,  $I_\alpha$ ($\alpha\in(0, N)$) is
the Riesz potential defined for every $x\in \R^N\setminus \{0\}$ by 
\begin{equation*}
I_\alpha(x) =A_\alpha|x|^{\alpha-N}, \quad \hbox{where }A_\alpha=\frac{\Gamma(\frac{N-\alpha}{2})}{2^\alpha\pi^{\frac{N}{2}}\Gamma(\frac{\alpha}{2})} \ \hbox{ and } \ \Gamma \ \hbox{is the Gamma function}
\end{equation*}
and $V(x)$ is the external potential, assumed to be, in our case, continuous  periodic. We recall that the choice of the constant $A_\alpha$  ensures the semigroup property 
\begin{equation}\label{sp}
I_\alpha \ast I_\beta = I_{\alpha+\beta}, \quad \forall \alpha, \beta > 0~\text{such~that}~\alpha + \beta < N.
\end{equation}
In the relevant physical case $N=3,  \alpha=N-2, p=2$  equation \eqref{Choqeq} takes its origin from  the so called Choquard-Pekar equation
\begin{equation*}
-\Delta u+u=\left(I_2\ast |u|^2\right)u \quad \text{for}~~x\in\R^3
\end{equation*}
 which appears in various  contexts, modeling the quantum polaron at rest \cite{fro, P1} and then used by Choquard in 1975, as pointed out by Lieb \cite{Lieb}, to study steady states of the one-component plasma approximation in the Hartree–Fock theory.  The Choquard-Pekar equation is also known as the Schr\"odinger-Newton equation in models coupling the Schr\"odinger equation of quantum physics together with nonrelativistic Newtonian gravity.
Lieb in \cite{Lieb} proved the existence and uniqueness of positive solutions to the Choquard equation by using rearrangement techniques. Multiplicity results  were then obtained by Lions \cite{Lions} by means of a variational
approach. A broad literature has been recently developed and we refer to \cite{CVZ,CZ,Du,W} for an up to date, though non exhaustive bibliography. We also refer to \cite{MS2} and references therein for an extensive survey on the topic.

Whereas  Choquard type problems have been widely investigated in the nonperiodic case, the case of  periodic (nonconstant) potentials $V(x)$ is much less studied: in this case the compactness issue  is  more difficult to handle due to the invariance of  the equation under the action of the noncompact group induced by translation in the coordinate directions. 
It is well known (see, e.g., \cite{RS}) that the spectrum of the self-adjoint operator $-\Delta +V$ in $L^2(\R^N)$ is purely continuous and may contain
gaps, i.e. open intervals free of spectrum. 
\vspace*{0.2cm}

The corresponding \emph{local} problem has been widely investigated, both for its physical applications and its mathematical  interest. In the case $\lambda =0$, the location of $0$ in the spectrum of $-\Delta+V$ determines the geometry of the associated energy functional. There are many results available when $\inf \sigma(-\Delta +V)>0$ (the {\it{positive definite}} case) or when $0$ lies in a gap of the spectrum  (the {\it{strongly indefinite}} case),  see, for example, \cite{CZR, Pankov1, Al, BJS, TW} or the monograph \cite{AM} and references therein. An interesting situation occurs when $0$ coincides with one of the borderline points of the spectral gap. As we will see later, this case is much more difficult to be approached, and several questions seem not yet solved. We refer to  \cite{BD} and  further generalizations proved in \cite{WZ, YCD, M, Sc, ACM}, which are, to the best of our knowledge, the only papers dealing with this case.  \\
The parameter dependent situation $\lambda \neq 0$ can be discussed replacing the potential $V(x)$ by $V(x)-\lambda$. In this context, an interesting physical and mathematical issue is establishing the existence of branches of solutions $u_\lambda$ converging towards the trivial solution as $\lambda$ approaches some point $\lambda_0$ of the spectrum, a so called \emph{bifurcation point}. An intriguing situation is the so called \emph{gap-bifurcation}, that is, bifurcation occurring at boundary points of the spectral gaps, when $V(x)$ is a periodic potential: the existence of nontrivial solutions  reveals the presence of bound states whose ``energy'' $\lambda\in\R$ lies in gaps of the spectrum of the
Schr\"{o}dinger operator $-\Delta+V$. These bound states are created by the nonlinear perturbation. A first systematic approach was obtained in \cite{KS, HKS}, see also \cite{S1, S2} and references therein. \\
In particular, in \cite{HKS}, Heinz, K\"{u}pper and Stuart applied some of their previous abstract results on bifurcation theory (see references in the paper) to study the following \emph{model} problem
\begin{equation*}\label{S1}
-\Delta u+V(x)u=r(x)|u|^{\sigma}u+\lambda u,\quad \text{in}~~\R^N,
\end{equation*}
where $r\in L^\infty(\R^N)$ and nonnegative a.e. on $\R^N$, with periodic continuous (nonconstant) potential $V(x)$. They  proved that, if $0<\sigma<\frac{4}{N-2}$ ($\sigma>0$ for $N=2$) and $[a,b]$ is a spectral gap for the Schr\"{o}dinger operator $-\Delta +V$, then there is a nontrivial solution for any $\lambda \in (a,b)$. Furthermore, if $r(x)$ is constant and $0<\sigma<4/N$, they proved that $b$ is a gap-bifurcation (from zero) point. In  \cite{KS1, S3} it is also shown that no bifurcation from 0 can occur in $a$, and  that the condition $0<\sigma<4/N$ is necessary for the displayed equation.

\vspace*{0.2cm}

What is known in the \emph{nonlocal} case, that is, for the Choquard equation \eqref{Choqeq}? The literature is much less developed. As in the local framework, in the particular  case $\lambda =0$ the geometry of the functional depends on the location of $0$ in the spectrum of $-\Delta+V$, and, consequently, the suitable approach.\\
When $V>0$ we are in the positive definite case, where  the spectrum lies in $(0,+\infty)$: Ackermann \cite{A} obtained a nontrivial solution via Mountain Pass Theorem as well as infinitely many geometrically distinct solutions, in the case of odd nonlinearities. \\
When the potential $V$ changes sign we are in the strongly indefinite case, and the spectrum consists of a union of closed intervals. If $\lambda =0$ is in a gap of the (essential) spectrum, the existence of  at least one nontrivial solution has been proved by Buffoni, Jeanjean and Stuart in \cite{BJS}, for the  physical case $N=3,  \alpha=N-2, p=2$, and by Ackermann \cite{A}, always in dimension $N=3$ but for a more general convolution kernel $W(x)$ and nonlinearity $f(u)$. The higher dimensional case has been recently approached by Qin, R\u{a}dulescu and Tang \cite{QRT}, who showed the existence of ground state solutions  without assuming the Ambrosetti-Rabinowitz type condition for $f(u)$.\\
The parameter dependent case and the related bifurcation analysis, instead, seem not to have been considered in literature. The first goal of the present paper is to address this issue. More precisely, if $V$ is a  potential satisfying the following conditions:\vspace*{0.3cm}
\begin{itemize}
	\item[$(V_1)$] $V \in C(\R^N)$ is $1-$periodic in $x_1,...x_N$;
	\item[$(V_2)$] $0$ lies in a gap of the spectrum of the Schr\"{o}dinger operator $-\Delta + V$, that is
	$$\sup[\sigma (-\Delta + V) \cap (-\infty, 0)]:=a < 0 < b:=\inf[\sigma (-\Delta + V) \cap (0,\infty)],$$
\end{itemize}
 we will prove
\begin{itemize}
	\item[$(1)$] the existence of  infinitely many {\it{geometrically distinct}} solutions $u_\lambda\in H^1(\R^N)$ for any $\lambda\in(a,b)$;
	\item[$(2)$] the convergence towards $0$ of these solutions $u_\lambda$, as $\lambda \to b^-$, for some values of $p$.
\end{itemize}Before giving the precise statements, we need to introduce the following definition.

\begin{definition}
	Suppose that $u_1, u_2$ solve \eqref{Choqeq}; if $\mathcal{O}(u_1)\neq\mathcal{O}(u_2)$,
	we say that $u_1$, $u_2 $ are geometrically distinct, where, $\mathcal{O}(u)$ denotes the orbit of $u_0$ with respect to the action of $\mathbb{Z}^N$:
	$$\mathcal{O}(u_0) := \{\tau_k u_0 : k \in \mathbb{Z}^N \},\quad (\tau_k u)(x) := u(x + k)$$
\end{definition}

We are now able to introduce our first result.
\begin{theorem}\label{th1} Let $N\geq3$, $\alpha\in(0,N)$, $p\in(\frac{N+\alpha}N,\frac{N+\alpha}{N-2})$ and $(V_1)-(V_2)$ hold. Then,
	for each $\lambda\in(a, b)$, there exist infinitely many geometrically distinct solutions $u_\lambda\in H^1(\R^N)$ of \eqref{Choqeq}.
\end{theorem}

The restrictions on the parameter $p$ follow from the variational approach: it guarantees  differentiability properties of the  energy functional. It is remarkable the appearance,  in the class of the Choquard type problems, of a lower nonlinear restriction.
\par \bigskip
The next result concerns the bifurcation from the right boundary point $b$ of the spectral gap, where we assume the following
\begin{definition} A point $\lambda\in\R$ is called a gap-bifurcation (from zero) point for \eqref{Choqeq} if there exists a
	sequence $\{(\lambda_n, u_n)\}$ of solutions of \eqref{Choqeq} such that $\lambda_n\in \rho(-\Delta+V)$, $\lambda_n\to\lambda$ and $\|u_n\|_{H^1(\R^N)}\to 0$ as $n \to\infty$,
	where $\rho(-\Delta+V) = \R\setminus\sigma(-\Delta+V)$ is the resolvent set for $-\Delta+V$.
\end{definition}

\begin{theorem}\label{th2} Let $N\geq3$, $\alpha\in(0,N)$, $p\in(\frac{N+\alpha}{N},\frac{N+\alpha}{N-2})$ and $(V_1)-(V_2)$ hold.
	For any $\lambda\in(a,b)$, let $u_\lambda$ be a solution of equation \eqref{Choqeq} obtained in Theorem \ref{th1}, then
	$$\Phi_\lambda(u_\lambda) = O((b-\lambda)^{\frac{2p-Np+N+\alpha}{2p-2}})\to0~~\text{as}~~\lambda\to b^-$$
	Moreover, if  $\frac{N+\alpha}{N}<p < \frac{N+\alpha+2}{N}$, then
	$$\|u_\lambda\|_{H^1} = O((b-\lambda)^{\frac{2-Np+N+\alpha}{4p-4}})\to0~~\text{as}~~\lambda\to b^-.$$
	Moreover, $b$ is the only possible gap-bifurcation (from zero) point for \eqref{Choqeq} in $[a,b]$.
\end{theorem}
 
The proofs are based on variational methods applied to the functional
\begin{equation}\label{phi}
\aligned
\Phi_\lambda(u)=\frac{1}{2}Q_\lambda(u)-\frac{1}{2p}J(u),
\endaligned
\end{equation}
where 
\begin{equation}\label{J0}
J(u):=\int_{\R^N}\left(I_\alpha\ast|u|^{p}\right)|u|^{p}dx
\end{equation}
The standard assumptions $\alpha \in (0,N)$ and $\frac{N+\alpha}N<p<\frac{N+\alpha}{N-2}$ imply that $\Phi_\lambda: H^1(\R^N)\to \R$ is of class $\mathcal C^1$ and that critical points of $\Phi_\lambda$ are weak solutions of \eqref{Choqeq}. By assumption $(V_2)$, we have $H^1 = E_{\lambda}^-\oplus E_\lambda^+$ corresponding to the decomposition of the spectrum: we will apply a generalized linking theorem due to Kryszewski-Szulkin, inspired by \cite{BD, A}. We emphasize that some of the results stated in the above theorems could also be obtained by applying abstract results in bifurcation theory (for example, from the results stated in Section 5 of \cite{S2}). However, verifying the validity of all the assumptions, in particular the so-called condition $T(\delta)$ is not trivial, and it would require the proofs of some intermediate lemmas. Hence, we prefer to give here a self-contained approach, referring to the cited papers for further generalizations.

We then address the interesting case $\lambda =a$, that is we study the Choquard-Schr\"odinger equation in a right boundary point of the spectrum  of the Schr\"odinger operator. As in the local framework, this location of $\lambda$  causes a loss of completeness in the decomposition of the domain of the Schr\"odinger operator. Inspired by  Bartsch and Ding \cite{BD},  we  find infinitely many geometrically distinct solutions for \eqref{Choqeq} which lie in $H^2_{loc}(\R^N)$ but not necessarily in $H^1(\R^N)$ (see Theorem \ref{th3}): the presence of the nonlocal term adds an extra difficulty in the choice of a suitable space framework and in the analysis of (PS) sequences. The results are stated in the following 

\begin{theorem}\label{th3} Let $N\geq3$, $\alpha\in(0,N)$, $p\in(\frac{N+\alpha}{N},\frac{N+\alpha}{N-2})$ and $(V_1)-(V_2)$ hold. For $\lambda=a$, \eqref{Choqeq} possess a nontrivial weak solution $u\in H^2_{loc}(\R^N)$.
Moreover, there exist infinitely many geometrically distinct solutions for \eqref{Choqeq} in $ H^2_{loc}(\R^N)$.
\end{theorem}

The paper is organized as follows. In Section 2, we introduce the variational framework and the main variational tools.  In Section 3 we address the existence result when $\lambda \in (a,b)$, proving Theorem \ref{th1}. In 
Section 4, we prove the bifurcation from zero at the right boundary point $b$, Theorem \ref{th2}. In Section 5 we consider the  delicate case $\lambda=a$, proving Theorem \ref{th3}. 

As stated in Theorem  \ref{th2}, $b$ is the {\it{only}} possible gap-bifurcation point for $0$ in $[a,b]$: what about the branches of solutions $u_\lambda$, as $\lambda \to a^+$? This is an intriguing question, which seems still open also in the local case. Some perspectives will be proposed in the final Section 6.

 \vskip2mm
 \par\noindent
 We make use of the following notation:  
 \begin{itemize}
 \item $|\cdot|_{p}$ denotes the usual norm of the space $L^{p}(\mathbb{R}^N)$.
 \item $C, C_i , i = 1, 2, \cdot\cdot\cdot,$ will be repeatedly used to denote various positive constants whose exact values are irrelevant. 
  \item $B_R(x_0)$ denotes the ball $\{x\in\R^N: |x-x_0|\leq R\}$. 
  \item $o(1)$ denotes the infinitesimal as $n\to+\infty$.
  \item For the sake of simplicity, integrals over the whole $\R^N$ will be often written $\int$.
  \item For the sake of simplicity, we often omit the constant $A_\alpha$ in $I_\alpha$.
 \end{itemize}
 
 \vskip4mm

\section{The functional framework}
 \setcounter{equation}{0}
 In this section we present the functional settings and regularity properties of the functionals defined in \eqref{phi}, \eqref{J0}. We start recalling  a classical inequality which turns out to be a main tool in our arguments. 
 
 \begin{proposition}\cite{LL} (Hardy--Littlewood--Sobolev inequality.) Let $t$, $r>1$ and $0<\alpha<N$ with $\frac{1}{t}+\frac{1}{r}=1+\frac{\alpha}{N}$, $f\in L^{t}(\R^N)$ and $h\in L^{r}(\R^N)$. There exists a constant $C(N,\alpha,t,r)$, independent of $f,h$, such that
 	\begin{equation}\label{NHLS}
 	|I_\alpha\ast h|_{t'}\leq C(N,\alpha,t,r)|h|_r
 	\end{equation}
 	and
 	\begin{equation}\label{HLS}
 	\int\left(I_\alpha\ast h\right)f\ dx\leq C(N,\alpha,t,r) |f|_{t}|h|_{r},
 	\end{equation}
 	where $|\cdot|_{s}$ denotes the $L^{s}(\mathbb{R}^{N})$-norm for $s\in[1,\infty]$, and $t'$ denotes the conjugate exponent such that $\frac{1}{t'}+\frac{1}{t}=1$.
 \end{proposition}
 
 Let us now prove some properties of the Riesz potential $I_\alpha$.
 \begin{proposition}\label{pro3}
 	Let $\alpha\in(0,N)$ and $f,g\in L^{\frac{2N}{N+\alpha}}(\R^N)$. Then
 	\begin{equation}\label{R0}
 	\int(I_\alpha\ast f)gdx=\int\left(I_{\frac{\alpha}{2}}\ast f\right)\left(I_{\frac{\alpha}{2}}\ast g\right)dx.
 	\end{equation}
 	Moreover,
 	\begin{equation}\label{R1}
 	\aligned
 	\int(I_\alpha\ast f)fdx=\int\left(I_{\frac{\alpha}{2}}\ast f\right)^2dx \geq0
 	\endaligned
 	\end{equation}
 	and
 	\begin{equation}\label{R}
 	\aligned
 	\int(I_\alpha\ast f)gdx\leq \left[\int(I_\alpha\ast f)fdx\right]^{\frac 12}\left[\int(I_\alpha\ast g)gdx\right]^{\frac 12}.
 	\endaligned
 	\end{equation}
 \end{proposition}
 \begin{proof}
 	By the semigroup property of the Riesz potential, \eqref{sp}, we have
 	$$\int(I_\alpha\ast f)gdx=\int(I_{\frac{\alpha}{2}}\ast I_{\frac{\alpha}{2}}\ast f)gdx=\int\int\int I_{\frac{\alpha}{2}}(y)I_{\frac{\alpha}{2}}(x-z-y)f(z)g(x)dxdydz.$$
 	Let $z=x'-y$, since $I_{\frac{\alpha}{2}}$ is even, we get
 	$$\int(I_\alpha\ast f)gdx=\int\int\int I_{\frac{\alpha}{2}}(y)I_{\frac{\alpha}{2}}(x'-x)f(x'-y)g(x)dxdydx'=
 	\int(I_{\frac{\alpha}{2}}\ast f)(I_{\frac{\alpha}{2}}\ast g)dx.$$
 	Thus \eqref{R0} holds. \eqref{R1} is obviously from \eqref{R0}.
 	
 	Using \eqref{R0} and the H\"{o}lder inequality, we obtain
 	\begin{equation*}
 	\aligned
 	\int(I_\alpha\ast f)gdx
 	=&\int\left(I_{\frac{\alpha}{2}}\ast f\right)\left(I_{\frac{\alpha}{2}}\ast g\right)dx \\
 	\leq &\left[\int\left(I_{\frac{\alpha}{2}}\ast f\right)^2dx\right]^{\frac{1}{2}}\left[\int\left(I_{\frac{\alpha}{2}}\ast g\right)^2dx\right]^{\frac{1}{2}}\\
 	= &\left[\int(I_\alpha\ast f)fdx\right]^{\frac{1}{2}}\left[\int(I_\alpha\ast g)gdx\right]^{\frac{1}{2}}.
 	\endaligned
 	\end{equation*}
 	This completes the proof.
 \end{proof}
 
Using the semigoup property of the Riesz potential, the nonlocal term \eqref{J0} in the energy functional can be written also as 
 		\begin{equation*}
 		J(u):=\int\left(I_\alpha\ast|u|^{p}\right)|u|^{p}dx=\int\left(I_\frac \alpha 2\ast|u|^{p}\right)^2dx.
 		\end{equation*}
 Our aim is now to define a \emph{natural} energy space associated with the energy functional $\Phi_\lambda$. As noted in the Introduction, the choice will depend on the location of $\lambda$ with respect to the spectrum of the Schr\"odinger operator. In what follows we state some properties of the nonlocal term $J(u)$, postponing to specific subsections the definitions and descriptions of the  functional frameworks. As part of the energy functional, the analysis  of $J(u)$ has been performed in several papers approaching variationally Choquard type equations. We mainly refer to the survey \cite{MS2}, and to \cite{MMVS}.
 
  It is easy to observe that $J$ is well defined on $H^1(\R^N)$, for any $\frac{N+\alpha}{N}\leq p \leq \frac{N+\alpha}{N-2}$: combining the  Hardy--Littlewood--Sobolev inequality \eqref{HLS} and the Sobolev inequality yields
 \begin{equation}\label{J1} 
 J(u)\leq C|u|^{2p}_{\frac{2Np}{N+\alpha}}\leq C\|u\|^{2p}_{H^1}.
 \end{equation}
 Furthermore, in \cite{MMVS} the authors noted that $J(u)$ is naturally settled in the so called \emph{Coulomb spaces} $\mathcal Q^{\alpha,p}$, defined as the vector spaces of measurable functions $u:\mathbb R^N\to \mathbb R$ such that $J(u)$ is finite. They also proved that the quantity
 \begin{equation}\label{Qap}
 \|u\|_{\mathcal Q^{\alpha, p}}:= \left(\int_{\mathbb R^N}\left|I_\frac \alpha 2\ast|u|^{p}\right|^2dx\right)^{\frac 1{2p}}
 \end{equation}
 defines a norm, which will guarantees the convexity of the functional $J$. Hence, inequality \eqref{J1} corresponds to the embedding  $H^{1}\subset L^{\frac{2Np}{N+\alpha}}\subset \mathcal Q^{\alpha,p}$. The paper \cite{MMVS} then introduces and carefully studies the Couloumb-Sobolev spaces and  regularity properties of $J$ in this framework, which differs from ours. Therefore, for the sake of completeness we state and prove some useful  properties of $J$ in $H^1$, even if some of them could be deduced by known results in literature.
  
 \begin{lemma}\label{J} The functional $J: H^1(\R^N)\mapsto\R$ satisfies the following properties:
 	\begin{itemize}
 		\item[$(i)$] $J$ is continuous and weakly sequentially lower semi-continuous.
 		\item[$(ii)$] For all $u, v\in H^1(\R^N)$, there is $C> 0$ such that
 		\begin{equation}\label{R3}
 		\langle J'(u), v\rangle\leq J(u)^{1-\frac{1}{2p}} J(v)^{\frac{1}{2p}}\leq C\|u\|^{2p-1}_{H^1}\|v\|_{H^1}.
 		\end{equation}
 		Moreover, $J'$ is weakly sequentially continuous. 
 		\item[$(iii)$] $J$ is even and convex, and for all $u, w\in H^1(\R^N)$
 		\begin{equation*}\label{R2}
 		J(u + w)\geq 2^{1-2p}J(u) - J(w).
 		\end{equation*}
 	\end{itemize}
 \end{lemma}
 \begin{proof} $(i)$ By \eqref{J1}, $J$ is is well defined. Let $u_n \to u$ in $H^1(\R^N)$,
 	then by the Hardy--Littlewood--Sobolev inequality and the elementary inequality $\big||a|^p-|b|^p\big|\leq |a-b|^p$,
 	$$\left|J(u_n)-J(u)\right|\leq |u_n-u|_{\frac{2Np}{N+\alpha}}|u_n|_{\frac{2Np}{N+\alpha}}+|u_n-u|_{\frac{2Np}{N+\alpha}}|u|_{\frac{2Np}{N+\alpha}}\to0,$$
 	as $n\to\infty$. Thus $J$ is continuous.
 	
 	Now let $u_n \rightharpoonup u$ in $H^1(\R^N)$. We
 	can assume (up to a subsequence) that $u_n \to u$ a.e. in $\R^N$.
 	By Fatou's Lemma,
 	$$J(u)=\int\int\lim\limits_{n\to\infty}\frac{|u_n(x)|^{p}|u_n(y)|^{p}}{|x-y|^{N-\alpha}}dxdy
 	\leq \liminf\limits_{n\to\infty}J(u_n)$$
 	Thus $J$ is weakly sequentially lower semi-continuous.
 	
 	$(ii)$ For any $u, v \in H^1(\R^N)$
 	\begin{equation*}
 	\langle J'(u),v\rangle = 2p\int\left(I_\alpha\ast|u|^{p}\right)|u|^{p-2}uvdx
 	\end{equation*}
 	By \eqref{R0}, \eqref{R1} and H\"{o}lder inequality  we have
 	\begin{multline}\label{e3.12}
 	\left|\int(I_\alpha\ast|u|^p)|u|^{p-2}u vdx\right|
 	\leq \left(\int(I_\alpha\ast|u|^p)|u|^{p}dx\right)^{\frac{1}{2}}\left(\int\left|I_{\frac{\alpha}{2}}\ast(|u|^{p-2}u v)\right|^2dx\right)^{\frac{1}{2}}\\
 	\leq \left(\int(I_\alpha\ast|u|^p)|u|^{p}dx\right)^{\frac{1}{2}} \left(\int\left| I_{\frac{\alpha}{2}}\ast |u|^p\right|^{2-\frac{2}{p}}\cdot\left|I_{\frac{\alpha}{2}}\ast|v|^p\right|^{\frac{2}{p}}\right)^{\frac 12}\\
	\leq \left(\int(I_\alpha\ast|u|^p)|u|^{p}dx\right)^{\frac{1}{2}}\left[\int\left|I_{\frac{\alpha}{2}}\ast|u|^p\right|^2dx\right]^{\frac{1}{2}-\frac{1}{2p}}
 	\left[\int\left|I_{\frac{\alpha}{2}}\ast|v|^p\right|^2dx\right]^{\frac{1}{2p}}\\
 	=\left(\int(I_\alpha\ast|u|^p)|u|^{p}dx\right)^{\frac{1}{2}}\left[\int(I_\alpha\ast|u|^p)|u|^pdx\right]^{\frac{1}{2}-\frac{1}{2p}}
 	\left[\int(I_\alpha\ast|v|^p)|v|^pdx\right]^{\frac{1}{2p}}\\
 	 \leq\left[J(u)\right]^{1-\frac{1}{2p}}
 	\left[J(v)\right]^{\frac{1}{2p}}
  	\end{multline}
 Hence, $J'$ is well defined.  	Let us now  prove that $J'$ is weakly sequentially continuous. Let us first show that if $u_n \rightharpoonup u$ in $H^1(\R^N)$, then
 	\begin{equation}\label{u1}
 	\int(I_\alpha\ast |u_n|^p)|u_n|^{p-2}u_nvdx-\int(I_\alpha\ast |u_n|^p)|u|^{p-2}uvdx\to0~\text{as}~n\to\infty
 	\end{equation}
 Indeed, since $v\in L^{\frac{2Np}{N+\alpha}}(\R^N)$, for any $\varepsilon > 0$ there is $R > 0$ such that $|v|_{L^{\frac{2Np}{N+\alpha}}(\R^N\setminus B_R)}\leq\varepsilon$. Then, by \eqref{R} 
 	 	\begin{multline}\label{u1bis}
 	\left|\int(I_\alpha\ast |u_n|^p)\left(|u_n|^{p-2}u_n-|u|^{p-2}u\right)v\chi_{\R^N\setminus B_R}dx\right| 	\leq \left|\int(I_\alpha\ast |u_n|^p)|u_n|^pdx\right|^{1/2}\cdot \\ \cdot \left|\int(I_\alpha\ast \left(|u_n|^{p-2}u_n-|u|^{p-2}u\right)v\chi_{\R^N\setminus B_R})\left(|u_n|^{p-2}u_n-|u|^{p-2}u\right)v\chi_{\R^N\setminus B_R}dx\right|^{1/2}\\
 	\leq C|u_n|^{1/2}_{\frac{2Np}{N+\alpha}}\cdot \left|\int(I_\alpha\ast \left(|u_n|^{p-2}u_n-|u|^{p-2}u\right)v\chi_{\R^N\setminus B_R})\left(|u_n|^{p-2}u_n-|u|^{p-2}u\right)v\chi_{\R^N\setminus B_R}dx\right|^{1/2}
 	\end{multline}
 	The estimate of the right hand side splits in two different cases - note that  $\{u_n\}$ is bounded in $L^{\frac{2Np}{N+\alpha}}(\R^N)$:
 	\begin{itemize}
 		\item If $p\geq 2$, then we apply the following inequality, which holds for any $a,b\in \R$:
 		$$
 		\left||a|^{p-2}a-|b|^{p-2}b\right|\leq C_p(|a|+|b|)^{p-2}|a-b|
 		$$
 		which yields
 		\begin{multline*}
 		 \left|\int(I_\alpha\ast \left(|u_n|^{p-2}u_n-|u|^{p-2}u\right)v\chi_{\R^N\setminus B_R})\left(|u_n|^{p-2}u_n-|u|^{p-2}u\right)v\chi_{\R^N\setminus B_R}dx\right|^{1/2}\\
 		 \leq C\left| \left(|u_n|^{p-2}u_n-|u|^{p-2}u\right)v\chi_{\R^N\setminus B_R}\right|_{\frac{2N}{N+\alpha}}\leq C_p\left| \left(|u_n|+|u|\right)^{p-2}|u_n-u||v|\chi_{\R^N\setminus B_R} \right|_{\frac{2N}{N+\alpha}}\\
 		 \leq C_p \left| (|u_n|+|u|) \right|^{p-2}_{\frac{2Np}{N+\alpha}}|u_n-u|_{\frac{2Np}{N+\alpha}}|v|_{L^{\frac{2Np}{N+\alpha}}(\R^N\setminus B_R)}\leq C\varepsilon
 		\end{multline*}
 			\item If $1<p < 2$, then we apply the following inequality, which holds for any $a,b\in \R$:
 		$$
 		\left||a|^{p-2}a-|b|^{p-2}b\right|\leq C_p|a-b|^{p-1}
 		$$
 		which yields
 		\begin{multline*}
 		\left|\int(I_\alpha\ast \left(|u_n|^{p-2}u_n-|u|^{p-2}u\right)v\chi_{\R^N\setminus B_R})\left(|u_n|^{p-2}u_n-|u|^{p-2}u\right)v\chi_{\R^N\setminus B_R}dx\right|^{1/2}\\
 		\leq C\left| \left(|u_n|^{p-2}u_n-|u|^{p-2}u\right)v\chi_{\R^N\setminus B_R}\right|_{\frac{2N}{N+\alpha}}\leq C_p\left| |u_n-u|^{p-1}|v|\chi_{\R^N\setminus B_R} \right|_{\frac{2N}{N+\alpha}}\\
 		\leq C_p |u_n-u|^{p-1}_{\frac{2Np}{N+\alpha}}|v|_{L^{\frac{2Np}{N+\alpha}}(\R^N\setminus B_R)}\leq C\varepsilon
 		\end{multline*}
 	\end{itemize}
 	In both cases, we obtain that for any fixed $v$ and for any $\varepsilon>0$ there is $R$ such that the right hand side of \eqref{u1bis} is less then $C\varepsilon$. By the same argument, on the ball $B_R$ we have:
 	\begin{itemize}
 		\item if $p\geq 2$
 		\begin{multline*}
 		\left|\int(I_\alpha\ast \left(|u_n|^{p-2}u_n-|u|^{p-2}u\right)v\chi_{B_R})\left(|u_n|^{p-2}u_n-|u|^{p-2}u\right)v\chi_{B_R}dx\right|^{1/2}\\
 		\leq C_p \left| (|u_n|+|u|) \right|^{p-2}_{\frac{2Np}{N+\alpha}}|u_n-u|_{L^{\frac{2Np}{N+\alpha}}( B_R)}|v|_{\frac{2Np}{N+\alpha}}\leq C\varepsilon
 		 		\end{multline*}
 		 		if $n\geq n_{\varepsilon}$, since $u_n\rightharpoonup u$ in $H^1$ and $p<\frac{N+\alpha}{N-2}$;
 		\item if $1<p < 2$, again,
 		\begin{multline*}
 		\left|\int(I_\alpha\ast \left(|u_n|^{p-2}u_n-|u|^{p-2}u\right)v\chi_{B_R})\left(|u_n|^{p-2}u_n-|u|^{p-2}u\right)v\chi_{B_R}dx\right|^{1/2}\\
 		\leq C_p |u_n-u|^{p-1}_{L^{\frac{2Np}{N+\alpha}}(B_R)}|v|_{\frac{2Np}{N+\alpha}}\leq C\varepsilon
 		\end{multline*}
 		if $n\geq n_{\varepsilon}$, since $u_n\rightharpoonup u$ in $H^1$ and $p<\frac{N+\alpha}{N-2}$.
 	\end{itemize}
 	Combining the above cases yields \eqref{u1}. Now, by \eqref{NHLS} the Riesz potential $I_\alpha$ defines a linear continuous map from $L^{\frac{2N}{N+\alpha}}(\R^N)$ to $L^{\frac{2N}{N-\alpha}}(\R^N)$.
 	Thus $(I_\alpha\ast |u_n|^p)\rightharpoonup (I_\alpha\ast |u|^p)$ in $L^{\frac{2N}{N-\alpha}}(\R^N)$, so that
 	\begin{equation}\label{u2}
 	\int(I_\alpha\ast |u_n|^p)|u|^{p-2}uvdx\to\int(I_\alpha\ast |u|^p)|u|^{p-2}uvdx~\text{as}~n\to\infty.
 	\end{equation}
 	since  $|u|^{p-2}uv\in L^{\frac{2N}{N+\alpha}}(\R^N)$. Combining \eqref{u1} and \eqref{u2} implies that  $J'$ is weakly sequentially continuous.
 	
 	$(iii)$ Obviously, $J$ is even. Proposition 2.1 in \cite{MMVS} proves that $\|u\|_{\mathcal Q^{\alpha, p}}$ is a norm, so that it is a convex functional. Since $J(u)=\|u\|_{\mathcal Q^{\alpha, p}}^{2p}$, it is also convex. By  convexity and the $2p$-homogeneity of $J$, we have
 	$$J(u + v) \leq \frac{1}{2}(J(2u)+J(2v))\leq 2^{2p-1}(J(u)+J(v))$$
 \end{proof}
 
We end stating a version of the nonlocal Br\'{e}zis-Lieb  property, see Lemma 2.4 in \cite{MS1} or the survey \cite{MS2} and references therein.
 \begin{lemma}\label{BLN}
 	Let $N \geq 3$, $0 < \alpha < N$ and $\frac{N+\alpha}{N}<p<\frac{N+\alpha}{N-2}$. Let $\{u_n\}\subset L^{\frac{2Np}{N+\alpha}}(\R^N)$ and $u_n\rightharpoonup u$ in $L^{\frac{2Np}{N+\alpha}}(\R^N)$, then
 	$$\int(I_\alpha\ast|u_n|^p)|u_n|^pdx-\int(I_\alpha\ast|u_n-u|^p)|u_n-u|^pdx\to \int(I_\alpha\ast|u|^p)|u|^pdx$$
 	as $n\to\infty$.
 \end{lemma} 

  In the following subsections we will give the details of the functional framework, which depends on the two different cases, $\lambda$ in the spectral gap $(a,b)$, or $\lambda$ in the right borderline point of the spectrum, that is $\lambda=a$.
 
\subsection{The case $\lambda \in(a,b)$}
Let $(a,b)$ denote a spectral gap as defined in assumption $(V2)$. For any $\lambda\in[a,b]$, let $S_\lambda := -\Delta + V-\lambda $. Under condition $(V_1)$, the operator $S_\lambda$ is self-adjoint in $L^2(\R^N)$ with domain $D(S_\lambda) = H^2(\R^N)$, and the spectrum $\sigma(S_\lambda)$ is purely absolutely continuous and bounded from below.
Let $\big(P_{\lambda,\nu} : L^2(\R^N) \to L^2(\R^N)\big)_{\nu\in\R}$ denote the spectral family of $S_\lambda$.
Setting $L^-_\lambda:= P_{\lambda,0}L^2(\R^N)$ and $L^+_\lambda := (Id - P_{\lambda,0})L^2(\R^N)$
we have the decomposition $L^2(\R^N)= L^-_\lambda \oplus L^+_\lambda$, where
 $L^-_\lambda\subset H^2(\R^N)$ since the spectrum of $S_\lambda$ is bounded from below.

 Let $E_\lambda$ be the completion of $D(\sqrt{|S_\lambda|})=H^1(\R^N )$ with respect to the norm
$$\|u\|_{E_\lambda} :=\left|\sqrt{|S_\lambda|}u\right|_2= \left(\int_{-\infty}^{+\infty}|\nu|d( P_{\lambda,\nu} u,u)_2\right)^{\frac{1}{2}},$$
where $|S_\lambda|$ is the absolute value of $S_\lambda$, such that
\begin{equation*}
|S_\lambda|u =
\left\{
\begin{array}{ll}
\aligned
S_\lambda u~~~~~&\text{if}~u \in D(S_\lambda) \cap L^+_\lambda; \\
-S_\lambda u~~~~~&\text{if}~u \in D(S_\lambda) \cap L^-_\lambda.
\endaligned
\end{array}
\right.
\end{equation*}
Clearly $E_\lambda$ is a Hilbert space with inner product
$$( u,v)_{E_\lambda} = \big(\sqrt{|S_\lambda|}u, \sqrt{|S_\lambda|}v\big)_2.$$
$E_\lambda$ can be  orthogonally decomposed as $E_\lambda = E^-_\lambda \oplus E^+_\lambda$, according to the
decomposition of $\sigma(S_\lambda)$. We shall write $u = u^- +u^+$ with $u^\pm \in E^\pm_\lambda$ for $u \in E_\lambda$, and
\begin{equation*}
\aligned
&\|u^+\|^2_{E_\lambda}=(S_\lambda u^+,u^+)_2=\int|\nabla u^+|^2 + (V (x)-\lambda)|u^+|^2dx, \\
&\|u^-\|^2_{E_\lambda}=-(S_\lambda u^-,u^-)_2= -\int|\nabla u^-|^2 + (V (x)-\lambda)|u^-|^2dx.
 \endaligned
\end{equation*}

For brevity, we denote $S_0=S$, $P_{0,\nu}=P_\nu$, $E_0=E$.
Note that, by assumptions $(V_1)-(V_2)$, $E = H^1(\R^N)$ and $\|\cdot\|_E$ is equivalent
to the usual norm of $H^1(\R^N)$.

Let $Q_\lambda: E_\lambda\to \R$ be the quadratic form
$$Q_\lambda(u) := \int|\nabla u|^2 + (V(x)-\lambda)u^2dx.$$
Then
$$Q_\lambda(u)=(S_\lambda u,u)_2=\|u^+\|^2_{E_\lambda}-\|u^-\|^2_{E_\lambda}\ ,$$
and $Q_0$ is negative definite on $E^-$ and positive definite on $E^+$ respectively, that is,
\begin{equation*}
Q_0(u^-) \leq -\alpha_0\|u^-\|^2_{H^1}, \quad Q_0(u^+) \geq \beta_0\|u^+\|^2_{H^1}
\end{equation*}
for all $u^- \in E^-$ and $u^+\in E^+$. Moreover, $$Q_0(u)=Q_0(u^- + u^+)= Q_0(u^-) + Q_0(u^+),$$
and the borderline points of the spectral gap  $(a, b)$ can be characterized as  
\begin{equation*}
a = \sup\limits_{u^-\in E^-, |u^-|_2=1}Q_0(u^-) < 0 < \inf\limits_{u^+\in E^+, |u^+|_2=1}Q_0(z) = b.
\end{equation*}
The same spectral splitting holds for any $\lambda\in (a, b)$. This is made precise by the following lemma.
\begin{proposition}\label{lm1} (\cite[Lemma 2]{T})
Let the spectral gap $(a,b)$ be given as in assumption $(V_2)$. 
Let $\lambda\in (a, b)$. Then
\begin{equation*}
Q_\lambda(u^-)=-\|u^-\|^2_{E_\lambda} \leq -\alpha_\lambda\|u^-\|^2_{H^1}, \quad Q_\lambda(u^+)=\|u^+\|^2_{E_\lambda} \geq \beta_\lambda\|u^+\|^2_{H^1}
\end{equation*}
for all $u^-\in E^-_\lambda$ and $u^+\in E^+_\lambda$, where
\begin{equation*}
\alpha_\lambda:=
\left\{
\begin{array}{ll}
\aligned
&\alpha_0\left(1 - \frac{\lambda}{a}\right)~~&\text{if}~~\lambda\leq0, \\
&\alpha_0~~&\text{if}~~\lambda>0,
\endaligned
\end{array}
\right.
\end{equation*}
\begin{equation*}
\beta_\lambda:=
\left\{
\begin{array}{ll}
\aligned
&\beta_0\left(1 - \frac{\lambda}{b}\right)~~&\text{if}~~\lambda>0, \\
& \beta_0 ~~&\text{if}~~\lambda\leq0.
\endaligned
\end{array}
\right.
\end{equation*}
Consequently,
\begin{equation*}\label{e1.10}
\|u\|^2_{E_\lambda}=Q_\lambda(u^+) - Q_\lambda(u^-) \geq \frac{1}{2} \min\{\alpha_\lambda, \beta_\lambda\}\|u\|^2_{H^1}.
\end{equation*}
\end{proposition}

As a consequence of Proposition \ref{lm1},  for any $\lambda\in(a,b)$ it holds $E_\lambda = E= H^1(\R^N)$.

\subsection{The case $\lambda=a$}
The case $\lambda=a$ requires a deeper analysis. In this situation $0\in\sigma(S_a)$ is a right boundary point of $\sigma(S_a)$, where $S_a=-\Delta+V-a\,$.
Since the spectrum of $S_a$ restricted to $L^+_a$ is contained in $[b-a, +\infty)$, which is bounded
away from $0$, the norm $\|\cdot\|_{E_a}$ is equivalent to the $H^1-$norm on $E^+_a$. However, $0$ is contained in the spectrum of $S_a$ restricted to $L^-_a$, hence the norm $\|\cdot\|_{E_a}$ is weaker than the $H^1$-norm on $E^-_a$. Moreover, $H^1(\R^N)\cap L^-_a=L^-_a$ is not complete with
respect to $\|\cdot\|_{E_a}$: indeed, since $0 \in \sigma(S_a)$ is a continuous spectrum point, there is a sequence $\{u_n\}\subset D(S_a)$ such that $|u_n|_2 = 1$ and $S_a u_n \to 0$, so that $\|u_n\|_{E_a} \to 0$.

Furthermore,  when $\lambda=a$, $J(u)$ is no more  well-defined on $E_a$.
To overcome these difficulties, we are going to define a new space $E_{HL}$ such that there are continuous
embeddings $H^1(\R^N )\subset E_{HL} \subset E_a$. Let us recall the definition of the Coulomb norm \eqref{Qap}
$$\|u\|_{\mathcal Q^{\alpha,p}}:=\left(\int \left(I_{\alpha}\ast |u|^p\right)|u|^p\right)^{\frac{1}{2p}}=\left(\int \left(I_{\frac \alpha 2}\ast |u|^p\right)^2\right)^{\frac{1}{2p}}$$
$\|\cdot\|_{\mathcal Q^{\alpha, p}}$ is a norm on $L^-_a\subset H^1(\R^N)$. Further, for any $u \in E^-_a$, we have
$$0 \leq \|u\|_{E_a}^2 =-\int (|\nabla u|^2 + (V (x)-a)u^2) dx$$
which implies
\begin{equation*}\label{G}
|\nabla u|_2 \leq C|u|_2,~~\forall u \in E^-_a.
\end{equation*}
Then, by Hardy-Littlewood-Sobolev inequality and Gagliardo-Nirenberg inequality, we have
\begin{equation*}\label{N}
\|u\|_{\mathcal Q^{\alpha, p}}\leq C|u|_{\frac{2Np}{N+\alpha}} \leq C|\nabla u|_2^{\frac{Np-N-\alpha}{2p}}|u|_2^{\frac{2p-Np+N+\alpha}{2p}}
\leq C|u|_2,~~\forall u \in E^-_a.
\end{equation*}
Let us define
$$ \|u\|_{E_\mathcal Q}:=\|u\|_{E_a}+\|u\|_{\mathcal Q^{\alpha, p}}$$
where, for the sake of brevity, we omit the indexes $a,\alpha, p$. Let $E_{\mathcal Q}^-$ be the completion of $L^-_a$ with respect to $\|\cdot\|_{E_\mathcal Q}$ and set $E_{\mathcal Q}:= E_{\mathcal Q}^- \oplus E^+_a$
Then $E_{\mathcal Q}$ is the completion of $H^1(\R^N)$ with respect to $ \|\cdot\|_{E_\mathcal Q}$ due to $E^+_a\sim L^+_a$.
Clearly $(E_{\mathcal Q}, \|\cdot\|_{E_\mathcal Q})$ is a Banach space and  $J(u)$ is well-defined on $E_{\mathcal Q}$.

\begin{remark}\label{remarkMMV}
	Let us stress the main difference between the local setting proposed in \cite{BD} and our choice. Mimicking \cite{BD}, we could choose as new space $E_{\frac{2Np}{N+\alpha}}$, defined as the completion of $H^1(\R^N)$ with respect to the norm 
	$$ \|\cdot\|_{E_a}+|\cdot|_{\frac{2Np}{N+\alpha}}
	$$ However, although the nonlinear term $(I_\alpha\ast |u|^p)|u|^{p}$ is well-defined in $E_{\frac{2Np}{N+\alpha}}$ by Hardy-Littlewood-Sobolev inequality, we are not able to prove that the (PS) sequences $\{u_n\}$ are bounded in $E_{\frac{2Np}{N+\alpha}}$. The main reason is that $\int(I_\alpha\ast |\cdot|^p)|\cdot|^{p}$ cannot control any Lebesgue norm $|\cdot|_{L^\mu(\R^N)}$, and we cannot count on $\|\cdot\|_{E_a}$ either, because the norm $\|\cdot\|_{E_a}$ is weaker than $\|\cdot\|_{H^1}$ in the singular case.
We overcome this difficulty taking into account the nonlocal nature of our  problem in the construction of the space $E_{\mathcal Q}$, which turns out to be embedded into $H_{loc}^1$, as we will prove in Lemma \ref{lm4.8} below.
 \end{remark}

Let us prove some basic properties of the space $E_\mathcal Q$.  
\begin{lemma}\label{norm}
	$H^1(\R^N) \subset E_{\mathcal Q} \subset E_a$ and all norms
	$\|\cdot\|_{E_a}$, $\|\cdot\|_{H^1}$, $\|\cdot\|_{E_\mathcal Q}$ are equivalent on $E^+_a$.
\end{lemma}
\begin{proof}  The embedding $E_{\mathcal Q} \subset E_a$ is obvious.
	By Sobolev embedding and Hardy-Littlewood-Sobolev inequality, for any $u\in H^1(\R^N)$, we have
	$$\|u\|_{H^1}\geq C|u|_{\frac{2Np}{N+\alpha}}\geq C\|u\|_{\mathcal Q^{\alpha,p}},$$
	On the other hand, we have $\|u\|_{H^1}\geq C\|u\|_{E_a}.$ Therefore
	$\|u\|_{H^1}\geq C\|u\|_{E_\mathcal Q}.$
	Thus $H^1(\R^N) \subset E_{\mathcal Q}$.\\For $u\in E^+_a$, we know that $\|\cdot\|_{H^1}$ and $\|\cdot\|_{E_a}$ are equivalent. Thus
	$$\|u\|_{H^1}\leq C\|u\|_{E_a}\leq C\|u\|_{E_\mathcal Q}.$$
	This completes the proof.
\end{proof}
We briefly recall that a norm  $\|\cdot\|_X$ on a linear space is said {\it{uniformly convex }} if for any $\varepsilon> 0$ there is a $\delta_\varepsilon>0$ such that for any $x,y\in X$ with $\|x\|_X=\|y\|_X=1$ and	$\|x-y\|_X\geq\varepsilon$, then $\|\frac{x+y}{2}\|_X\leq1-\delta_\varepsilon.$ In \cite{MMVS}, Proposition 2.8, the authors prove the following property:

\begin{lemma}\label{lm4.4}[Proposition 2.8 in \cite{MMVS}]
Let $\alpha \in (0,N)$ and $p>1$. Then	$\|\cdot\|_{\mathcal Q^{\alpha,p}}$ is a uniformly convex norm. 
\end{lemma}
Consequently, $E_{\mathcal Q}$ is a reflexive Banach space.\\As already said in the introduction, the location of $\lambda$ on the right borderline point of the spectrum prevents the embedding of $E^-_a$ in $H^1$.  Nevertheless, we can recover a partial regularity, stated in the following Lemma.
\begin{lemma}\label{lm4.8} $E_{\mathcal Q}^-$ embeds continuously into $H^1_{loc}(\R^N)$, and hence compactly into
	$L^t_{loc}(\R^N)$ for $2 \leq t < 2^*$. Moreover, $S_au \in L^2$ for $u\in E_{\mathcal Q}^-$, and $E_{\mathcal Q}^-$ embeds continuously into $H^2_{loc}(\R^N)$.
\end{lemma}
\begin{proof}
Let us first prove the embedding of $E_{\mathcal Q}^-$ in $ H^1_{loc}(\R^N)$. Let $u \in E_{\mathcal Q}^-$. Since $L_a^-$ is dense in $E_{\mathcal Q}^-$, we can choose a sequence $\{u_n\}_{n\in\mathbb{N}}$  in $L_a^-$ with $\|u_n - u\|_{E_\mathcal Q} \to 0$, as $n \to\infty$. \\	
For any fixed $R\in \R^+$ and for any $x\in B_{R/2}(0)$, we have $B_R(x)\supset B_{R/2}(0)$ so that
\begin{multline}\label{ploc}
\|u\|_{\mathcal Q^{\alpha, p}}^{2p}=\int|u(x)|^p\int\frac{|u(y)|^p}{|x-y|^{N-\alpha}}dy	\geq \frac{1}{R^{N-\alpha}}\int|u(x)|^p\left[\int_{B_R(x)}|u(y)|^pdy\right]dx\\
\geq \frac{1}{R^{N-\alpha}}\int_{B_{R/2}(0)}|u(x)|^p\left[\int_{B_R(x)}|u(y)|^pdy\right]dx \\
\geq \frac{1}{R^{N-\alpha}}\int_{B_{R/2}(0)}|u(x)|^p\left[\int_{B_{R/2}(0)}|u(y)|^pdy\right]dx
=\frac{1}{R^{N-\alpha}}\left(\int_{B_{R/2}(0)}|u(x)|^pdx\right)^2
\end{multline}
Therefore, $u\in L_{loc}^p(\R^N)$.\\	
Given a bounded domain $\Omega\subset\R^N$,
let us take a function $\eta\in C_0^\infty(\R^N)$ with $\eta \equiv 1$ in $\Omega$. Then for any $v \in L_a^-\subset H^2(\R^N)$,
\begin{equation*}
-\Delta(\eta v)\eta v =\eta^2\cdot (-\Delta v) \cdot v + v^2 \cdot (-\Delta \eta)\eta - 2 \nabla (\eta v) \cdot v  \nabla \eta+2|\nabla\eta|^2v^2,
\end{equation*}
so that we get
\begin{multline}\label{etav}
\int |\nabla(\eta v)|^2dx 
\leq \langle S_av,\eta^2v\rangle-\int(V(x)-a)\eta^2v^2dx+\int v^2 \cdot (-\Delta \eta)\eta dx+\\+\frac{1}{2}\int|\nabla(\eta v)|^2dx+4\int|\nabla\eta|^2v^2dx
\leq C\|v\|_{E_a}^2+C|\eta v|_2^2+\frac{1}{2}\int|\nabla(\eta v)|^2dx
\end{multline}
so that:
\begin{itemize}
	\item if $p\geq 2$, combining \eqref{ploc} with \eqref{etav} yields immediately
	\begin{equation*}
	\int |\nabla(\eta v)|^2dx \leq C\|\eta v\|_{E_a}^2+C|\eta v|_{2}^2 \leq  C\|\eta v\|_{E_a}^2+C|\eta v|_{p}^2 \leq  C\|\eta v\|_{E_a}^2+C\|\eta v\|_{\mathcal Q^{\alpha, p}}^2
	\end{equation*}
	where $C$ depends on $\Omega$.
\item if $p < 2$ we combine the interpolation inequality with Young inequality and \eqref{ploc}:
\begin{multline*}
|\eta v|_{2}^2\leq C|\eta v|_{2^\ast}^{2\theta} |\eta v|_{p}^{2(1-\theta)} \quad \hbox{ where } \ \ \theta=\frac{N(2-p)}{2N-p(N-2)}\\
\leq \varepsilon |\nabla(\eta v)|_2^2+C_{\varepsilon}|\eta v|_{p}^2\leq  \varepsilon |\nabla(\eta v)|_2^2+C_{\varepsilon}\|\eta v\|_{\mathcal Q^{\alpha, p}}^2
\end{multline*} 
Inserting this last inequality into \eqref{etav} and choosing $\varepsilon$ small enough we obtain, again,
\begin{equation*}
\int_{\Omega} |\nabla(\eta v)|^2dx \leq  C\|\eta v\|_{E_a}^2+C\|\eta v\|_{\mathcal Q^{\alpha, p}}^2
\end{equation*}
	where $C$ depends on $\Omega$.
\end{itemize}

We have then proved that,  for any $v \in L_a^-\subset H^2(\R^N)$ 
\begin{equation*}
\int_{\Omega} |\nabla(\eta v)|^2dx \leq  C\|\eta v\|_{E_a}^2+C\|\eta v\|_{\mathcal Q^{\alpha, p}}^2
\end{equation*}
where $\Omega$ is any bounded domain and $\eta\in C_0^\infty(\R^N)$ with $\eta \equiv 1$ in $\Omega$. Applying this inequality to the above sequence $\{u_n\}$ we obtain that $\{u_n\}$ is a Cauchy sequence in $H^1(\Omega)$, and hence $u\in H^1_{loc}$. Thus  $E_{\mathcal Q}^-$ embeds continuously into $H^1_{loc}(\R^N)$.\\
Now we can follow the same lines of the proof of \cite[Lemma 2.1]{BD} to show that $S_au \in L^2$ and $u\in H^2_{loc}(\R^N)$.	For the convenience of the reader, we give the details. Since $\inf \sigma(S_a):=-\theta > -\infty$ we have
\begin{multline*}
	|S_a(u_n - u_m)|_2^2 = \int_{-\theta}^0 \nu^2d |P_{a,\nu}(u_n - u_m)|_2^2 
	\leq -\theta\int^0_{-\theta}\nu d |P_{a,\nu}(u_n - u_m)|_2^2 \\
	= \theta\left||S_a|^{\frac{1}{2}}(u_n - u_m)\right|_2^2 = \theta\|u_n - u_m\|^2_{E_a}.
\end{multline*}
Therefore $\{S_au_n\}$ is a Cauchy sequence in $L^2$ and it follows that $S_au_n \to S_au$ in $L^2$.\\	
For $r > 0$, $\varepsilon> 0$ and $y \in \R^N$, by Calderon-Zygmund inequality \cite[Theorem 9.11]{GT} we have
$$\|u_n - u_m\|_{H^2(B_r(y))} \leq C(r,\varepsilon) \left( |u_n - u_m|_{L^2(B_{r+\varepsilon}(y))} + |S_a(u_n- u_m)|_{L^2(B_{r+\varepsilon}(y))}\right).$$
This implies $u \in H^2_{loc}(\R^N)$.
\end{proof}
\begin{remark}
	A space closely related to ours has been introduced by Ruiz in \cite{Ru} in the more relevant  physical   case $N=3, \alpha =2$.\\
	We observe that another possible choice for the functional setting in the case $\lambda=a$ could be a variant of the Coulomb-Sobolev spaces introduced in \cite{MMVS}.  
\end{remark}

{\section{Existence of solutions for $\lambda\in(a,b)$. }}
 \setcounter{equation}{0}
The aim of this section is to prove Theorem \ref{th1}. As discussed in the previous section, if $a<\lambda<b$, then $\|\cdot\|_{E_\lambda}$ is equivalent to
$\|\cdot\|_{H^1}$ and $E_\lambda= H^1(\R^N)$.\\
Due to the geometry of the functional $\Phi_\lambda$, the main tool to find nontrivial critical points will be the following generalized
Linking Theorem \cite{BD,D}.

\begin{theorem}[Generalized Linking Theorem \cite{BD,D}]\label{LT}Let $X$ be a real Hilbert space. Suppose that $\Phi\in C^1(X, \R)$
satisfies the following conditions:
\begin{itemize}
\item[$(i)$] There exists a bounded self-adjoint linear operator $L : X \mapsto X$ and
a functional $\Psi \in C^1(X, \R)$ which is bounded below, weakly sequentially
lower semi-continuous with $\Psi': X \mapsto X$ weakly sequentially continuous
and such that
$$\Phi(u) = \frac{1}{2} \langle Lu,u\rangle-\Psi(u).$$
\item[$(ii)$] There exist a closed separable L-invariant subspace $Y$ of $X$ and a
positive constant $\alpha$ such that
$$\langle Lu, u\rangle \leq -\alpha \|u\|^2_X \ \hbox{for } u \in Y\ \  \hbox{and} \ \ \langle Lu, u\rangle \geq \alpha \|u\|^2_X \ \hbox{for } u \in Z:=Y^{\bot}.$$
\item[$(iii)$] There are constants $\kappa, \rho > 0$ such that $\Phi(u) \geq \kappa$ for $u \in Z$ and $\|u\|_X=\rho$.
\item[$(iv)$] Let $\zeta\in Z\setminus\{0\}$.
Then there exists $R>\rho$ ($R$ depending on $\zeta$) such that $\Phi(u) \leq 0$ for any $u \in \partial M$, where
$$M := \{u = u^- + s\zeta : u^- \in Y, s \geq 0~~\text{and}~~\|u\|_X\leq R\}.$$
\end{itemize}
Then there exists a Palais-Smale sequence $\{u_n\}$ such that $$\Phi(u_n) \to c\in [\kappa, \sup \Phi(M)]~~\text{and}~~\Phi'(u_n) \to 0,\quad\text{as}~n\to\infty.$$
\end{theorem}
Let us now verify the our functional  $\Phi_\lambda$ verifies the linking structure of the above theorem, assumptions $(iii)$ and $(iv)$.
\begin{lemma}\label{lm2.2} For any $\lambda\in(a,b)$,
there exist $\kappa(\lambda)$, $\rho > 0$ such that for any $u \in E^+_\lambda\cap \partial B_\rho(0)$ it results that
$\inf\limits_{u \in E^+_\lambda, \|u\|_{E_\lambda}=\rho}\Phi_\lambda(u):=\kappa(\lambda)>0$.
\end{lemma}
\begin{proof} By Proposition \ref{lm1}, for any $u \in E^+_\lambda\setminus\{0\}$ we have $\|u^+\|_{H^1}^{2}\leq\frac{1}{\beta_\lambda}\|u^+\|_{E_\lambda}^{2}.$
Then, by Sobolev embedding and Hardy-Littlewood-Sobolev inequality, we have
\begin{equation*}
\Phi_\lambda(u)\geq \frac{1}{2}\|u^+\|^2_{E_\lambda}-C(N,\alpha,p)\|u^+\|_{H^1}^{2p}
\geq\frac{1}{2}\|u^+\|^2_{E_\lambda}-\frac{C(N,\alpha,p)}{\beta_\lambda^p}\|u^+\|^{2p}_{E_\lambda}.
\end{equation*}
Let $\rho =\frac{1}{2}\left(\frac{\beta_\lambda^p}{2C(N,\alpha,p)}\right)^{\frac{1}{2p-2}}$; since $p>1$, we have
\begin{equation}\label{ka2}
\kappa(\lambda):= \Phi_\lambda(u)|_{E^+_\lambda\cap \partial B_\rho(0)}\geq\left(\frac{1}{8}-\frac{1}{2^{2p+1}}\right)
\left(\frac{\beta_\lambda^p}{2C(N,\alpha,p)}\right)^{\frac{1}{p-1}}>0.
\end{equation}
\end{proof}

\begin{lemma}\label{lm2.3}
Let $Z_0$ be a finite dimensional subspace of $E^+_a$. Then $\Phi_\lambda(u) \to -\infty$ as
$\|u\|_{E_\lambda}\to\infty$ in $E^-_\lambda\oplus Z_0$.
\end{lemma}
\begin{proof} Following \cite[Lemma 4.2]{A}, for $\beta\in (0, 1)$, we set $\gamma = \sin(\arctan \beta) \in (0, 1)$ and
$$K= \{u \in E_\lambda: u^+\in Z_0, \|u^+\|_{E_\lambda} \geq \gamma, \|u\|_{E_\lambda}= 1\}.$$
Then there is $\{u_n\} \subset K$
with $\lim\limits_{n\to\infty}J(u_n) = \inf J(K) =: \delta \geq 0$. Since $K$ is bounded we may assume
that $u_n \rightharpoonup u \in E_\lambda$ such that $u^+_n \to u^+$ in $Z_0$. Clearly $\|u^+\|_{E_\lambda}\geq \gamma$ and $u\neq0$. Since $J$ is weakly sequentially lower semi-continuous in $E_\lambda$, we have
$\delta \geq J(u) > 0.$

Let $u\in E^-_\lambda\oplus Z_0$ satisfy $\|u\|_{E_\lambda}\geq1$. We have two cases.\\
$\bullet$ If $\|u^+\|_{E_\lambda}/\|u^-\|_{E_\lambda}\geq\beta$ we have
$$\|u^+\|_{E_\lambda} /\|u\|_{E_\lambda} = \sin\arctan(\|u^+\|_{E_\lambda} /\|u^-\|_{E_\lambda})\geq\gamma$$
and therefore $u/\|u\|_{E_\lambda}\in K$. By $J(u)=J(u/\|u\|_{E_\lambda})\|u\|_{E_\lambda}^{2p}$ and the definition of $\delta$ we obtain
$J(u) \geq \delta \|u\|_{E_\lambda}^{2p}$ and
$$\Phi_\lambda(u) \leq \frac{1}{2} \|u\|_{E_\lambda}^2 -\frac{ \delta}{2p}\|u\|_{E_\lambda}^{2p}.$$
$\bullet$ If $\|u^+\|_{E_\lambda}/\|u^-\|_{E_\lambda}<\beta$ we have
\begin{equation}\label{e2.2}
\Phi_\lambda(u) \leq \frac{1}{2}(\|u^+\|^2_{E_\lambda} - \|u^-\|^2_{E_\lambda}) \leq -\frac{1 -\beta^2}{ 2(1 + \beta^2)} \|u\|^2_{E_\lambda}.
\end{equation}
For $\|u\|_{E_\lambda}$ large we find in either case that (\ref{e2.2}) is satisfied, and the claim is proved
since $\beta^2 < 1$.
\end{proof}


By Lemma \ref{J} and by Lemma \ref{lm2.3} $\Phi_\lambda$ satisfies all the conditions in Theorem \ref{LT}, for any $\lambda\in(a,b)$.
Thus, there exists a Palais-Smale sequence $\{u_n\}$ at level $c_\lambda$,
\begin{equation}\label{c}
c_\lambda\in[\kappa(\lambda), \sup\limits_{u\in E^-_\lambda \oplus \R^+\zeta}\Phi_\lambda],
\end{equation}
where $\kappa(\lambda)>0$ is a constant that depends on $\lambda$.
Moreover, by Proposition \ref{lm1} and \eqref{ka2}, we have
$$\kappa(\lambda)\geq\left(\frac{1}{8}-\frac{1}{2^{2p+1}}\right)
\left(\frac{\beta_\lambda^p}{2C(N,\alpha,p)}\right)^{\frac{1}{p-1}}\to 0,\quad\text{as}~\lambda\to b^-.$$
In the following lemma we verify the boundedness of any (PS) sequence.
\begin{lemma}\label{lm2.1}
 If $\{u_n\}$ is a $(PS)_{c_\lambda}-$sequence for $\Phi_\lambda$.
 Then $\|u_n\|_{E_\lambda}$ are bounded.
\end{lemma}
\begin{proof} Let $n$ large such that $\Phi_\lambda(u_n)\leq c_\lambda+1$ and $\|\Phi'_\lambda(u_n)\|_{E_\lambda}\leq1$. Then
\begin{equation}\label{20}
 \aligned
c_\lambda+1+\frac{1}{2}\|u_n\|_{E_\lambda}\geq\Phi_\lambda(u_n)-\frac{1}{2}\langle\Phi'_{\lambda}(u_n),u_n\rangle
=\left(\frac{1}{2}-\frac{1}{2p}\right)J(u_n).
  \endaligned
\end{equation}
By \eqref{e3.12} and \eqref{20}, we have
\begin{multline*}\label{23}
\|u_n^+\|_{E_\lambda}^2=\langle\Phi'_{\lambda}(u_n),u_n^+\rangle+\int(I_\alpha\ast|u_n|^p)|u_n|^{p-2}u_nu_n^+dx \\
\leq 1\cdot\|u_n^+\|_{E_\lambda}+J(u_n)^{1-\frac{1}{2p}}J(u_n^+)^{\frac{1}{2p}}
\leq \|u_n^+\|_{E_\lambda}+C(\lambda)(1+\|u_n\|_{E_\lambda})^{1-\frac{1}{2p}}\|u_n^+\|_{E_\lambda}
\end{multline*}
Thus
\begin{equation*}
 \aligned
\|u_n^+\|_{E_\lambda}^2\leq C(\lambda)(1+\|u_n\|_{E_\lambda})^{2-\frac{1}{p}},
  \endaligned
\end{equation*}
which together with
$$ \| u^-_n\|^2_{E_\lambda} \leq -2\Phi_\lambda(u_n) + \| u^+_n\|^2_{E_\lambda}$$
implies that
\begin{equation*}
 \aligned
\|u_n\|_{E_\lambda}^2=\| u^+_n\|^2_{E_\lambda}+\| u^-_n\|^2_{E_\lambda}\leq C(\lambda)(1+\|u_n\|_{E_\lambda})^{2-\frac{1}{p}}.
  \endaligned
\end{equation*}
Since $2-\frac{1}{p}<2$, $\|u_n\|_{E_\lambda}$ is bounded.
\end{proof}

By the previous arguments we have obtained a $(PS)_{c_\lambda}$-sequence $\{u_n\}$ which is bounded in $E_\lambda$. Then by using Lions' concentration compactness principle \cite[Lemma 1.21]{Willem} and the invariance of $\Phi_\lambda$ under the action of $\mathbb{Z}^N$, we get a nontrivial weak solution for \eqref{Choqeq}. Similar to \cite{A}, by using Theorem 4.2 in \cite{BD}, the existence of infinitely many geometrically distinct solutions can be obtained in a similar way.  Thus we have proved Theorem \ref{th1}.

\vskip4mm
{\section{ Bifurcation from zero when $\lambda\to b^-$. }}
 \setcounter{equation}{0}
In this section we prove Theorem \ref{th2}, that is, the bifurcation phenomenon occurring on the left borderline point of the spectrum, for some values of $p$, extending to the nonlocal case the results present in \cite{HKS, T}.\\
Since $b \in \sigma(-\Delta+V)$, we know that there exists a Bloch wave $\Psi$ in $H^2_{\text{loc}}(\R^N) \cap C^1(\R^N) \cap L^\infty(\R^N)$ that satisfies $-\Delta \Psi + V \Psi = b\Psi$ (see \cite{E}).
$\Psi$ is uniformly almost-periodic (UAP) in the sense of Besicovitch \cite{B}.
The essential tool is a nonlocal version of the Riemann-Lebesgue lemma and the estimate of the nonlocal part of the functional $\Phi_\lambda( \Psi_{(b-\lambda)^{-1/2}})$:  $$\int\int\frac{|\Psi_{(b-\lambda)^{-1/2}}(x)|^p|\Psi_{(b-\lambda)^{-1/2}}(y)|^p}{|x-y|^{N-\alpha}}dxdy,$$
where the testing vectors $\Psi_{(b-\lambda)^{-1/2}}$ are constructed from the Bloch wave $\Psi$ of the linear Schr\"{o}dinger operator.

To any uniformly almost-periodic (UAP) function $f: \R^N \to \mathbb{C}$ is associated a mean-value, $M(f)$, which may be defined by
$$M(f)=\lim\limits_{T\to\infty}\frac{1}{T^N}\int_0^T\cdot\cdot\cdot\int_0^Tf(x)dx_1\cdot\cdot\cdot dx_N.$$
We recall here the classical Riemann-Lebesgue lemma.
\begin{proposition}\label{p2.1} (\cite{HKS})
 Let $f: \R^N \to \mathbb{C}$ be a uniformly almost-periodic (UAP) function and let $g\in L^1(\R^N)$.
Then
\begin{equation*}
\lim\limits_{T\to\infty} \int f(Tx)g(x)dx = M(f)\int g(x) dx.
\end{equation*}
\end{proposition}
For $R \in (0, +\infty)$, let us set
$$\Psi_R(x) := R^{-\frac{N}{2}} \eta\left(\frac{x}{R}\right) \Psi(x)$$
where $\eta\in C^\infty_0(\R^N ; [0, 1])$ equals $1$ on $B(0, 1)$. Then, $\Psi_R\in H^2(\R^N)\cap C^1(\R^N)$.\\
It is easy to see from Proposition \ref{p2.1} that for all $\gamma\in[1,+\infty)$,
\begin{equation*}\label{2.2}
\lim\limits_{R\to\infty}R^{\frac{N}{2}-\frac{N}{\gamma}}|\Psi_R|_\gamma=[M(\Psi^\gamma)]^{\frac{1}{\gamma}}|\eta|_\gamma.
\end{equation*}
The following  proposition states a nonlocal version of the Riemann-Lebesgue lemma, which is an easy consequence of the classical one.
\begin{lemma}\label{lm3.2}
 Let $f: \R^N \to \mathbb{C}$ he a  uniformly almost-periodic (UAP) function and let $g\in L^{\frac{2N}{N+\alpha}}(\R^N)$.
Then
\begin{equation*}\label{e3.0}
\lim\limits_{T\to\infty} \int\int\frac{f(Tx)g(x)f(Ty)g(y)}{|x-y|^{N-\alpha}}dxdy = [M(f)]^2\int\int\frac{g(x)g(y)}{|x-y|^{N-\alpha}} dxdy.
\end{equation*}
\end{lemma}
\begin{proof} Since $$\left|\int\int\frac{g(y)g(x)}{|x-y|^{N-\alpha}}dxdy\right|\leq\left(\int|g|^{\frac{2N}{N+\alpha}}dx\right)^{\frac{N+\alpha}{N}}<\infty $$ and $f(Tx)f(Ty)$ is a uniformly almost-periodic (UAP) function, then by Proposition \ref{p2.1}, we get the conclusion.
\end{proof}
Let us now apply the above lemma  to estimate the functional $J$ tested on the Bloch wave $\Psi$:
\begin{multline*}
J(\Psi_{R})=\int\int\frac{|\Psi_{R}(x)|^p|\Psi_{R}(y)|^p}{|x-y|^{N-\alpha}}dxdy =R^{-Np}\int\int\frac{|\eta\left(\frac{x}{R}\right) \Psi(x)|^p|\eta\left(\frac{y}{R}\right) \Psi(y)|^p}{|x-y|^{N-\alpha}}dxdy \\
=R^{N+\alpha-Np} \int\int\frac{|\eta\left(x\right) \Psi(Rx)|^p|\eta\left(y\right) \Psi(Ry)|^p}{|x-y|^{N-\alpha}}dxdy,
\end{multline*}
then by Lemma \ref{lm3.2}, we get 
\begin{equation}\label{lm3.3}
\lim\limits_{R\to\infty}R^{Np-N-\alpha}J(\Psi_{R})=[M(|\Psi|^p)]^2J(\eta).
\end{equation}
Now, for $\lambda\in (a, b)$, let $R(\lambda) := \frac{1}{\sqrt{b - \lambda}}$.
From \cite{T}, we know that
\begin{equation}\label{e3.1}
\|P_0\Psi_{R(\lambda)}\|_{H^1}=O(b-\lambda)~\text{as}~\lambda\to b.
 \end{equation}
 By the Hardy--Littlewood--Sobolev inequality and Sobolev inequality, we have
\begin{equation*}
J(P_0\Psi_{R(\lambda)})\leq C|P_0\Psi_{R(\lambda)}|^{2p}_{\frac{2Np}{N+\alpha}}\leq C\|P_0\Psi_{R(\lambda)}\|^{2p}_{H^1},
 \end{equation*}
 which together with (\ref{e3.1}) implies that
 \begin{equation}\label{e3.2}
 J(P_0\Psi_{R(\lambda)})=O(|b-\lambda|^{2p})~\text{as}~\lambda\to b.
  \end{equation}
Let us define
$$\zeta_\lambda:= (Id - P_0)\Psi_{R(\lambda)} \in E^+$$
By Lemma \ref{J}-(iii), we have
$$J(\zeta_\lambda)=J(\Psi_{R(\lambda)}-P_0\Psi_{R(\lambda)})\geq 2^{1-2p}J(\Psi_{R(\lambda)})-J(P_0\Psi_{R(\lambda)})$$
Since $p<\frac{N+\alpha}{N-2}$, we also have $2p-\frac{Np-N-\alpha}{2}>0$; hence, combining the last inequality with \eqref{lm3.3} and \eqref{e3.2}  yields
 \begin{equation}\label{e3.3}
\liminf\limits_{\lambda\to b}(b - \lambda)^{-\frac{Np-N-\alpha}{2}}J(\zeta_\lambda)\geq 2^{1-2p}[M(|\Psi|^p)]^2J(\eta)>0
\end{equation}
On the other hand, we have from \cite{S1,T} that
\begin{equation}\label{3.6}
 Q_\lambda(\zeta_\lambda) = O(b - \lambda)~~\text{as}~~\lambda\to b^-
 \end{equation}
We are now ready to prove the first part of Theorem \ref{th2}, that is an estimate for the critical level $c_\lambda$ found in \eqref{c}, as $\lambda \to b^-$.
\begin{proposition}\label{pro1}
 $c_\lambda= O((b-\lambda)^{\frac{2p-Np+N+\alpha}{2p-2}})\to0$ as $\lambda\to b^-.$
\end{proposition}
\begin{proof}
By \eqref{c}, we have
$$\kappa(\lambda)\leq c_\lambda\leq \sup\limits_{v\in E^-_\lambda, s\geq0} \Phi_\lambda(v+s\zeta_\lambda)
=\sup\limits_{v\in E^-_\lambda, s\geq0}\left[ \frac{1}{2}Q_\lambda(v) + \frac{1}{2} s^2Q_\lambda(\zeta_\lambda) - \frac{1}{2p}J(v + s\zeta_\lambda)\right],$$
where $Q_\lambda(v)\leq0$, and from Lemma \ref{J}-$(iii)$
$$J(v + s\zeta_\lambda) >s^{2p}\left( 2^{1-2p}J(\zeta_\lambda)-J\left(v/s\right)\right). $$
Let us now prove that
\begin{equation}\label{3.21}
J\left(v/s\right)\leq C\alpha_\lambda^{-p}| Q_\lambda(\zeta_\lambda)|^p.
\end{equation}
Indeed, since
$$\sup\limits_{v\in E^-,~s> 0} \Phi_\lambda(v + s\zeta_\lambda)\geq c_\lambda > 0,$$
 we can restrict our attention to the couples $(v, s)$  satisfying $\Phi_\lambda(v + s\zeta_\lambda) \geq 0$ and $s > 0$
Then by Proposition \ref{lm1}, we get
\begin{equation*}\label{3.2}
\Phi_\lambda(v + s\zeta_\lambda)\geq0 \Rightarrow Q_\lambda(v) + s^2Q_\lambda(\zeta_\lambda) \geq 0
\Rightarrow  Q_\lambda(\zeta_\lambda)\geq \alpha_\lambda\left\|v/s\right\|^2_{H^1}.
\end{equation*}
On the other hand, by the Hardy--Littlewood--Sobolev inequality and the Sobolev inequality, we have
\begin{equation*}\label{3.1}
J\left(v/s\right)\leq C(N,p,\alpha)\left|v/s\right|^{2p}_{\frac{2Np}{N+\alpha}} \leq C\left\|v/s\right\|^{2p}_{H^1}
\end{equation*}
Combining the two above inequalities yields \eqref{3.21}.\\
From \eqref{3.21}, we deduce
$$\Phi_\lambda(v+s\zeta_\lambda)\leq\frac{1}{2}s^2 Q_\lambda(\zeta_\lambda)-\frac{1}{2p}s^{2p}\left[2^{1-2p}J(\zeta_\lambda)-C\alpha_\lambda^{-p}|Q_\lambda(\zeta_\lambda)|^p\right]$$
Since
$$\frac{N+\alpha}{N}\leq p<\frac{N+\alpha}{N-2}\Rightarrow0<\frac{Np-N-\alpha}{2}<p,$$
then by \eqref{lm3.3}, \eqref{e3.3} and \eqref{3.6}, for $\lambda$ approaching $b$ sufficiently ($\lambda<b$),
$$L(\zeta_\lambda):=\frac{1}{2p}\left[2^{1-2p}J(\zeta_\lambda)-C\alpha_\lambda^{-p}|Q_\lambda(\zeta_\lambda)|^p\right]>0,$$
and $$L(\zeta_\lambda)=O((b-\lambda)^{\frac{Np-N-\alpha}{2}})$$
Therefore,
\begin{equation*}
 \aligned
c_\lambda\leq\sup\limits_{s>0}\left(\frac{1}{2}s^2 Q_\lambda(\zeta_\lambda)-s^{2p}L(\zeta_\lambda)\right)
=&(p-1)(2p)^{-\frac{2p}{2p-2}}Q_\lambda(\zeta_\lambda)^{\frac{2p}{2p-2}}L(\zeta_\lambda)^{-\frac{2}{2p-2}} \\
=&O((b-\lambda)^{\frac{2p-Np+N+\alpha}{2p-2}})
\endaligned
\end{equation*}
Since
$$\frac{N+\alpha}{N}<p<\frac{N+\alpha}{N-2}\Rightarrow\frac{2p-Np+N+\alpha}{2p-2}>0,$$
we have the final consequence
$$c_\lambda\to 0~\text{as}~\lambda\to b^-.$$
\end{proof}

Let $\{u_n\}$ be a Palais-Smale sequence at level $c_\lambda$ such that $u_n \rightharpoonup u_\lambda$ in $H^1(\R^N)$.
The weak limit $u_\lambda$ is a critical point of $\Phi_\lambda$.
By Lemma \ref{BLN} (the Br\'{e}zis-Lieb lemma of nonlocal type), we have
\begin{equation}\label{cl}
 \aligned
 c_\lambda= \ &  \Phi_\lambda(u_n)-\frac{1}{2}\langle\Phi'_\lambda(u_n),u_n\rangle+o(1) = \frac{p-1}{2p}\int\left(I_\alpha\ast|u_n|^{p}\right)|u_n|^{p}dx+o(1) \\
 = \ & \frac{p-1}{2p}\int\left(I_\alpha\ast|u_\lambda|^{p}\right)|u_\lambda|^{p}dx
 +\frac{p-1}{2p}\int\left(I_\alpha\ast|u_n-u_\lambda|^{p}\right)|u_n-u_\lambda|^{p}dx+o(1)\\
 \geq   \ &\frac{p-1}{2p}\int\left(I_\alpha\ast|u_\lambda|^{p}\right)|u_\lambda|^{p}dx+o(1) \\
 = \ &\Phi_\lambda(u_\lambda)-\frac{1}{2}\langle\Phi'_\lambda(u_\lambda),u_\lambda\rangle+o(1)=\Phi_\lambda(u_\lambda)+o(1).
 \endaligned
\end{equation}
Combining this estimate with Proposition \ref{pro1} we can prove the second part of Theorem \ref{th2}.
\begin{proposition}\label{pro2}
When $\lambda\to b^-$, $\|u_\lambda\|_{H^1} = O(\sqrt{c_\lambda/\beta_\lambda})$,
and in particular
$$\|u_\lambda\|_{H^1} = O((b-\lambda)^{\frac{2-Np+N+\alpha}{4p-4}})\to0,$$
if $\frac{N+\alpha}{N}\leq p < 1 + \frac{2+\alpha}{N}$.
\end{proposition}
\begin{proof} Let us apply \eqref{cl} to get a relationship between $c_\lambda$ and $u_\lambda$:
\begin{equation*}
 \aligned
c_\lambda\geq \Phi_\lambda(u_\lambda)
=\Phi_\lambda(u_\lambda)-\frac{1}{2}\langle\Phi'_\lambda(u_\lambda),u_\lambda\rangle
=\frac{p-1}{2p}J(u_\lambda)
 \endaligned
\end{equation*}
Decompose $u_\lambda$ as $u_\lambda = u_\lambda^- + u_\lambda^+$ with
$u_\lambda^- \in E^-_\lambda$, $u_\lambda^+\in E^+_\lambda$. Proposition \ref{lm1} and $\Phi'_\lambda(u_\lambda) = 0$ imply that
\begin{multline}\label{3.11}
\beta_\lambda\|u_\lambda^+\|^2_{H^1} + \alpha_\lambda\|u_\lambda^- \|^2_{H^1}\leq Q_\lambda(u_\lambda^+) - Q_\lambda(u_\lambda^-) \\
= \frac{1}{2}\langle Q'_\lambda(u_\lambda), u_\lambda^+-u_\lambda^-\rangle 
=  \int(I_\alpha\ast |u_\lambda|^p)|u_\lambda|^{p-2}u_\lambda(u_\lambda^+-u_\lambda^-)dx\\
= J(u_\lambda)-2\int(I_\alpha\ast |u_\lambda|^p)|u_\lambda|^{p-2}u_\lambda u_\lambda^-dx
\end{multline}
From \eqref{R3} we have
\begin{equation*}
\aligned
\int(I_\alpha\ast|u_\lambda|^p)|u_\lambda|^{p-2}u_\lambda u_\lambda^-dx
\leq C\left[J(u_\lambda)\right]^{1-\frac{1}{2p}} \| u_\lambda^-\|_{H^1},
\endaligned
\end{equation*}
so that, by \eqref{3.11} and Young inequality, we get
\begin{equation}\label{3.17}
 \aligned
\beta_\lambda\|u_\lambda^+\|^2_{H^1} + \frac{\alpha_\lambda}{2}\|u_\lambda^- \|^2_{H^1}
\leq&\frac{2p}{p-1}c_\lambda+Cc_\lambda^{2-\frac{1}{p}}
 \endaligned
\end{equation}
Now let $\lambda\to b^-$: then $\alpha_\lambda=\alpha_0$, $\beta_\lambda=\frac{\beta_0}{b}(b-\lambda)$ and $c_\lambda\to0$.
By \eqref{3.17}, we have
$$\lim\limits_{\lambda\to b^-}\frac{ (b-\lambda)\|u_\lambda\|^2_{H^1}}{c_\lambda} \leq C   $$
Therefore, by Proposition \ref{pro1}, we obtain
$$\|u_\lambda\|_{H^1}= O(\sqrt{c_\lambda/\beta_\lambda})= O((b-\lambda)^{\frac{2-Np+N+\alpha}{4p-4}})$$
Moreover, if $\frac{N+\alpha}{N}\leq p < 1 + \frac{2+\alpha}{N}$, then $\frac{2-Np+N+\alpha}{4p-4}>0$ and
$\lim\limits_{\lambda\to b^-}\|u_\lambda\|_{H^1}=0.$
\end{proof}

Combining Proposition \ref{pro1} and Proposition \ref{pro2}, we have proved the first and the second part part of Theorem \ref{th2}. Now, let us complete the proof, verifying that  $b$ is the only possible gap-bifurcation point for \eqref{Choqeq} in $[a,b]$.

Let $u_\lambda=u^-_\lambda+u^+_\lambda$ be a nontrivial weak solution of \eqref{Choqeq}. Then by testing \eqref{Choqeq} with $u^+_\lambda$ and $u^-_\lambda$, we have
\begin{equation*}
Q_\lambda(u^+_\lambda)=\int(I_\alpha\ast|u_\lambda|^p)|u_\lambda|^{p-2}u_\lambda u_\lambda^+\quad \hbox{and} \quad 
Q_\lambda(u^-_\lambda)=\int(I_\alpha\ast|u_\lambda|^p)|u_\lambda|^{p-2}u_\lambda u_\lambda^-
\end{equation*}
which implies directly
\begin{equation*}
 \aligned
Q_\lambda(u^+_\lambda)-Q_\lambda(u^-_\lambda)=\frac{1}{2p}\langle J'(u_\lambda), u_\lambda^+-u^-_\lambda\rangle=\frac{1}{2p}\langle J'(u_\lambda), 2u_\lambda^+-u_\lambda\rangle
 \endaligned
\end{equation*}
Since $J : H^1(\R^N) \mapsto\R$ is even and convex, $\langle J'(u_\lambda), 2u_\lambda^+-u_\lambda\rangle\leq J(2u^+_\lambda)-J(u_\lambda).$
Then, by Hardy--Littlewood--Sobolev inequality and Sobolev inequality, we get
\begin{equation*}
 \aligned
Q_\lambda(u^+_\lambda)-Q_\lambda(u^-_\lambda)
\leq\frac{1}{2p} J(2u^+_\lambda)-\frac{1}{2p}J(u_\lambda)\leq \frac{2^{2p}}{2p} J(u_\lambda^+)
\leq\frac{2^{2p}}{2p} \|u_\lambda^+\|^{2p}_E
 \endaligned
\end{equation*}
By Proposition \ref{lm1}, we have
\begin{equation*}
 \aligned
\beta_\lambda\|u^+_\lambda\|^2_{H^1}+\alpha_\lambda\|u^-_\lambda\|^2_{H^1}\leq \frac{2^{2p}}{2p}\|u^+_\lambda\|^{2p}_{H^1} \quad \Longrightarrow \quad  \|u^+_\lambda\|^{2p-2}_{H^1}\geq \frac{2p\beta_\lambda}{2^{2p}}
  \endaligned
\end{equation*}
Therefore, by the definition of $\beta_\lambda$,  $b$ is the only possible gap-bifurcation point for \eqref{Choqeq} in $[a,b]$.

{\section{The case $\lambda= a$: existence of $H^1_{loc}$-solutions. }}

In this section we focus on the most delicate case, that is, when $\lambda=a$, the right borderline point of the spectrum of our Schr\"{o}dinger operator. We will prove Theorem \ref{th3}.

In Section 2 we defined the space $E_{\mathcal Q}$ and proved some of its properties. Although $J$ is well-defined on $E_{\mathcal Q}$, $E_{\mathcal Q}$ is not a Hilbert space and, due to the location of $\lambda$, $\Phi_a$ does not present a linking structure as required by the generalized Linking Theorem \ref{LT} (in particular, condition (ii) is not verified). To overcome this problem, we use an approximation argument like \cite{BD}.
For each $j\in \mathbb{N}$ we set
$$E_j^- := P_{a,-1/j}L^2(\R^N) = P_{a,-1/j}L^-_a\subset L^-_a \subset E^-_a$$
and
$$E_j:= E_j^- \oplus E^+_a \subset E_a$$
Then the spectrum of $S_a$ restricted to each  $E_j$ is bounded away from $0$, and we have
$$ \|\cdot\|_{E_a} \sim \|\cdot\|_{H^1}~\text{on}~E_j$$
Let
$$Q_j := P_{a,-1/j} + (Id - P_{a,0}) : E_a \mapsto E_j$$
denote the orthogonal projection. Then we have for any $u \in H^1(\R^N)$:
$$ Q_ju \to u~\text{as}~j \to\infty,~\text{with~respect~to}~\|\cdot\|_{E_a} ~\text{and}~ | \cdot |_t,~2 \leq t < 2^*$$

For each $j\in\mathbb{N}$, let $\Phi_j:=\Phi_a|_{E_j}$, $J_j:=J|_{E_j}$, where $E_j=E_j^-\oplus E^+_a$, $E^-_j=P_{a,-1/j}L^2(\R^N)$.
Obviously, $\Phi_j, J_j\in C^1(E_j,\R)$ and for $u,v\in E_j$,
\begin{equation*}
 \aligned
&\langle J_j'(u),v\rangle=\int(I_\alpha\ast |u|^p)|u|^{p-2}uvdx, \\
&\langle\Phi'_j(u),v\rangle=\langle S_au,v\rangle-\int(I_\alpha\ast |u|^p)|u|^{p-2}uvdx
  \endaligned
\end{equation*}

\begin{definition} A sequence $\{u_n\}_{n\in\mathbb{N}}$ is said to be a $\widetilde{(PS)}_c$-sequence for $\Phi_a$ with
respect to $(E_{j_n}, \|\cdot\|_{E_a})$, some $c \in \R$, if\\
(i) $u_n \in E_{j_n}$ with $j_n \to\infty $ as $n \to\infty$; \\
(ii) $\Phi_a(u_n ) \to c$ as $n \to\infty$; \\
(iii) $ \|\Phi'_{j_n}(u_n)\|_{E_a} \to 0$ as $n\to\infty$.
\end{definition}
Let us first prove the boundedness of the $\widetilde{(PS)}_c$-sequences.
\begin{lemma}\label{lm4.1}
 If $\{u_n\}$ is a $\widetilde{(PS)}_c$-sequence for $\Phi_a$, then $\|u_n\|_{E_a}$ and $\|u_n\|_{\mathcal Q^{\alpha, p}}$ are
bounded or equivalently, $\|u_n\|_{E_\mathcal Q}$ is bounded.
\end{lemma}
\begin{proof} Let $n$ large such that $\Phi_a(u_n)\leq c+1$ and $\|\Phi'_{j_n}(u_n)\|_{E_a}\leq1$,
then
\begin{equation}\label{40}
 \aligned
c+1+\frac{1}{2}\|u_n\|_{E_a}\geq\Phi_a(u_n)-\frac{1}{2}\langle\Phi'_{j_n}(u_n),u_n\rangle=\left(\frac{1}{2}-\frac{1}{2p}\right)|u_n|_{HL}^{2p}
  \endaligned
\end{equation}
Note that by \eqref{e3.12}, we have
\begin{equation}\label{42}
 \aligned
\left|\int(I_\alpha\ast|u_n|^p)|u_n|^{p-2}u_nu_n^+dx\right|\leq \|u_n\|_{\mathcal Q^{\alpha,p}}^{2p-1}\|u_n^+\|_{\mathcal Q^{\alpha,p}}
  \endaligned
\end{equation}
Thus, by \eqref{40}-\eqref{42}, we have
\begin{equation}\label{43}
 \aligned
\|u_n^+\|_{E_a}^2=&\langle\Phi'_{j_n}(u_n),u_n^+\rangle+\int(I_\alpha\ast|u_n|^p)|u_n|^{p-2}u_nu_n^+dx \\
\leq&1\cdot\|u_n^+\|_{E_a}+\|u_n\|_{\mathcal Q^{\alpha,p}}^{2p-1}\|u_n^+\|_{\mathcal Q^{\alpha,p}}\\
\leq&\|u_n^+\|_{E_a}+[C(1+\|u_n\|_{E_a})]^{1-\frac{1}{2p}}\|u_n^+\|_{\mathcal Q^{\alpha,p}}
  \endaligned
\end{equation}
Since $\|u_n^+\|_{E_a}$ and $\|u_n^+\|_{H^1}$ are equivalent for $u_n^+\in E^+_a$, we have
$$\|u_n^+\|_{\mathcal Q^{\alpha,p}}\leq C|u_n^+|_{\frac{2Np}{N+\alpha}}\leq C\|u_n^+\|_{H^1}\leq C\|u_n^+\|_{E_a}$$
Then by \eqref{43}, we have
\begin{equation*}
\|u_n^+\|_{E_a}^2\leq \|u_n^+\|_{E_a}+[C(1+\|u_n\|_{E_a})]^{1-\frac{1}{2p}}\|u_n^+\|_{E_a} \quad \Longrightarrow \quad \|u_n^+\|_{E_a}^2\leq [C(1+\|u_n\|_{E_a})]^{2-\frac{1}{p}},
\end{equation*}
which together with
$$ \| u^-_n\|^2_{E_a} \leq -2\Phi_a(u_n) + \| u^+_n\|^2_{E_a}$$
implies that
\begin{equation*}
 \aligned
\|u_n\|_{E_a}^2=\| u^+_n\|^2_{E_a}+\| u^-_n\|^2_{E_a}\leq C +[C(1+\|u_n\|_{E_a})]^{2-\frac{1}{p}}
  \endaligned
\end{equation*}
Since $2-\frac{1}{p}<2$, $\|u_n\|_{E_a}$ is bounded, hence,
applying \eqref{40} once more we obtain that $\|u_n\|_{\mathcal Q^{\alpha,p}}$ is bounded.
\end{proof}
Let us note that $\Phi_j \in C^1(E_j, \R)$ has the form $\Phi_j(u) = \frac{1}{2}\langle S_au, u\rangle_{E_a}
-J(u)$. From Lemma \ref{J} and the fact that $\|\cdot\|_{E_a}$ and $\|\cdot\|_{H^1}$ are equivalent on $E_j$, we can deduce that 
$J\in C^1(E_j, \R)$ is bounded below, weakly sequentially lower semicontinuous
and $\nabla_{E_a}J: E_j \mapsto {E_j}$ is weakly sequentially continuous.
Obviously, the functional $\Phi_j$ satisfies the conditions $(i)-(ii)$ in Theorem \ref{LT}. Following the same lines as for the proofs of Lemma \ref{lm2.2} and Lemma \ref{lm2.3}, we can verify that the functional $\Phi_j$ satisfies the linking structure, that is, conditions $(iii)-(iv)$ in Theorem \ref{LT} as stated in the following
\begin{lemma}\label{lm2}
There exist $\kappa$, $\rho > 0$ such that for any $u \in S^+_\rho:= E^+_a\cap \partial B_\rho(0)$ it results that
$\inf\Phi_j(S^+_\rho):=\kappa >0$.
\end{lemma}

\begin{lemma}\label{lm3}
Let $Z_0$ be a finite dimensional subspace of $E^+_a$. Then $\Phi_j(u) \to -\infty$ as
$\|u\|_{E_a}\to\infty$ in $E^-_j\oplus Z_0$.
\end{lemma}
%
For the sake of brevity, we omit the two proofs.

Setting $X := E_j$, $Y := E_j^-$ and $Z := E^+_a$, by Lemma \ref{lm2} and Lemma \ref{lm3}, $\Phi_j$ satisfies all the assumptions of Theorem \ref{LT}. Consequently, for any $j$ there exists a sequence $\{v^j_m\}_{m\in\mathbb{N}}$
in $E_j$ such that $\Phi'_j(v^j_m) \to 0$ and $\Phi_j(v^j_m) \to c_j \in [\kappa, \sup\Phi_j(M)]$ as $m \to\infty$, where
$\kappa>0$ is defined in Lemma \ref{lm2}, and
 $M$ is defined as
 $$M := \{u = u^- + s\zeta : u^- \in E^-_j, s \geq 0~~\text{and}~~\|u\|_{E_a}\leq R\}$$
For $m(j)$ large we therefore have
$$\|\Phi'_j(v^j_m(j))\|_{E_a} + |c_j - \Phi_j(v^j_{m(j)})| < \frac{1}{j}.$$
Since $$\sup\limits_{M}\Phi_j(u)\leq \frac{1}{2}\sup\limits_{M}\left[\|u^+\|_{E_a}^2-\|u^-\|_{E_a}^2\right]\leq \frac{1}{2}\sup\limits_{M}\left[\|u\|_{E_a}^2\right]\leq\frac{1}{2}R^2,$$
there is a subsequence $c_{j_n}$ such that  $c_{j_n} \to c \in [\kappa, \frac{1}{2}R^2]$. The sequence $u_n:= v^j_{m(j_n)}$
is then a $\widetilde{(PS)}_c$-sequence as required. By Lemma \ref{lm4.1}, $\{u_n\}$ is bounded in $E_{\mathcal Q}$.  Since $E_{\mathcal Q}$ is a reflexive Banach space by Lemma \ref{lm4.4}, up to a subsequence we have $u_n\rightharpoonup u$ in $E_{\mathcal Q}$.

Let us now show that $u\neq0$. We claim that for any $r > 0$ there exists a sequence $\{y_n\}$ in $\R^N$ and $\delta> 0$ such that
\begin{equation}\label{e4.6}
\liminf\limits_{n\to\infty}\int_{B_r(y_n)}u_n^2dx\geq\delta.
\end{equation}
Indeed, if not, then by Lions' concentration compactness principle \cite[Lemma 1.21]{Willem},
$$u_n \to 0~\text{in}~ L^q(\R^N)~\forall  q\in(2, 2^*),\quad \text{as}~n\to\infty.$$
Then, by Hardy-Littlewood-Sobolev inequality, we have
\begin{equation}\label{e4.7}
J(u_n)\leq C|u_n|^{2p}_{\frac{2Np}{N+\alpha}}\to0,\quad \text{as}~n\to\infty.
\end{equation}
On the other hand, we have
$$0<c=\Phi_j(u_n)-\frac{1}{2}\langle\Phi'_j(u_n),u_n\rangle+o(1)=\left(\frac{1}{2}-\frac{1}{2p}\right)J(u_n)+o(1),$$
which contradicts \eqref{e4.7}. Thus, \eqref{e4.6} holds. Now we choose $k_n \in \mathbb{Z}^N$ such that
$$|k_n - y_n | = \min \{|k- y_n | : k \in \mathbb{Z}^N\} $$	 and
set $v_n:= u_n(\cdot+k_n)$. Using \eqref{e4.6} and the invariance of ${E_{j}}_n
, E^\pm$ under the action of $\mathbb{Z}^N$ we see that $v_n\in {E_{j}}_n$ and
\begin{equation}\label{e4.8}
\int_{B_{r+\sqrt{N}/2}(0)}v_n^2dx\geq\frac{\delta}{2}.
\end{equation}Moreover, $\|v_n\|_{E_a} = \|u_n\|_{E_a}$ and $\|v_n\|_{\mathcal Q^{\alpha, p}} = \|u_n\|_{\mathcal Q^{\alpha,p}}$, hence $\|v_n\|_{E_\mathcal Q}$ is bounded. Lemma \ref{lm4.4} yields the existence of a subsequence (which we continue to denote by $\{v_n\}$)
such that $v_n \rightharpoonup u$ weakly in $E_{\mathcal Q}$. Then by Lemma \ref{lm4.8}, $v_n \to u$ strongly in $L^2_{loc}(\R^N)$. Clearly \eqref{e4.8} implies $u \neq 0$.

Let $v \in C_0^\infty(\R^N)$ be any test function. By Hardy-Littlewood-Sobolev inequality and H\"{o}lder inequality, we see that
\begin{equation*}
 \aligned
\int(I_\alpha\ast|v_n|^p)|v_n|^{p-2}v_n(Id - {Q_j}_n )v dx
\leq &|v_n^p|_{\frac{2N}{N+\alpha}} |v_n^{p-1}(Id - {Q_j}_n )v |_{\frac{2N}{N+\alpha}}\\
\leq &|v_n|^p_{\frac{2Np}{N+\alpha}} |v_n|^{p-1}_{\frac{2Np}{N+\alpha}}|(Id - {Q_j}_n )v|^p_{\frac{2Np}{N+\alpha}}
  \endaligned
  \end{equation*}
The right hand side converges to $0$ as $n \to\infty$. Now
\begin{equation*}
 \aligned
&\langle S_av_n, v\rangle_{E_a} =\langle S_av_n, {Q_j}_nv\rangle_{E_{a}} \\
&= \langle \Phi'_a(v_n), {Q_j}_nv\rangle + \int (I_\alpha\ast|v_n|^p)|v_n|^{p-2}v_nv dx
-\int(I_\alpha\ast|v_n|^p)|v_n|^{p-2}v_n(Id - {Q_j}_n )v dx
  \endaligned
  \end{equation*}
and therefore, letting $n \to\infty$, we have
$$\int(\nabla u \cdot \nabla v + (V (x)-a)uv) dx = \langle S_au, v\rangle_{E_a} = \int (I_\alpha\ast|u|^p)|u|^{p-2}uv dx $$
This shows that $u$ is a weak solution for \eqref{Choqeq}.

We end this section by proving the multiplicity result for \eqref{Choqeq}: it will be a consequence of Theorem 4.2 in \cite{BD}, see also \cite{A}. Let us first recall the definition of $(PS)_I$-attractor:
\begin{definition}
	Let $\Phi: X\mapsto\R$, denote $\Phi^b_a = \{ u \in X: a \leq \Phi(u) \leq b \}$.
	Given an interval $I \subset \R$, call a set $\mathcal{A} \subset X$ a $(PS)_I$-attractor if for any $(PS)_c$-sequence $\{u_n\}$ with $c \in I$, and any $\varepsilon, \delta > 0$ one has
	$u_n \in U_\varepsilon(\mathcal{A} \cap \Phi^{c+\delta}_{c-\delta})$ provided $n$ is large enough.
\end{definition}

\begin{theorem}\label{th4.6} (\cite{BD}). Let $X$ be a reflexive Banach space with the direct
sum decomposition $X = X^-\oplus X^+$, $u = u^- +u^+$ for $u \in X$, and suppose that
$X^-$ is separable.
If $\Phi$ satisfies the following hypotheses:
\begin{itemize}
\item[$(\Phi_1)$] $\Phi\in C^1(X, \R)$ is even and $\Phi(0) = 0$.
\item[$(\Phi_2)$] There exist $\kappa, \rho> 0$ such that $\Phi(z) \geq \kappa$ for every $z \in X^+$ with $\|z\|_X = \rho$.
\item[$(\Phi_3)$] There exists a strictly increasing sequence of finite-dimensional subspaces
$Z_n \subset X^+$ such that $\sup \Phi(X_n) < \infty$ where $X_n := X^- \oplus Z_n$, and an increasing sequence of real numbers $r_n > 0$ with $\Phi(X_n\setminus B_{r_n} ) < \inf \Phi(B_\rho)$.
\item[$(\Phi_4)$] $\Phi(u)\to-\infty$ as $\|u^-\|_X\to\infty$ and $\|u^+\|_X$ bounded.
\item[$(\Phi_5)$] $\Phi': X_w^-\oplus X^+\to X_w^*$ is sequentially continuous, and $\Phi: X_w^- \oplus X^+ \to \R$ is
sequentially upper semi-continuous, where $X_w^-$ denote the space $X^-$ with the weak topology.
\item[$(\Phi_6)$] For any compact interval $I \subset (0,\infty)$ there exists a $(PS)_I-$attractor $\mathcal{A}$ such
that $$\inf\{ \|u^+ - v^+\|_X : u, v \in\mathcal{A},~u^+ \neq v^+ \} > 0.$$
\end{itemize}
Then there exists an unbounded sequence $(c_n)$ of positive critical values.
\end{theorem}

Write $$\mathcal{K} =: \{u \in E_{\mathcal Q}: \Phi'_a(u) = 0\}$$ for the set of critical points.
 Let $\mathcal{F}$ consist of arbitrarily chosen representatives of the orbits in $\mathcal{K}$ under the action of
$\mathbb{Z}^N$. By the evenness of $\Phi_a$ we can also assume that $\mathcal{F} = -\mathcal{F}$. To prove that there
are infinitely many geometrically distinct solutions of \eqref{Choqeq}, setting $X:= E_{\mathcal Q}$, $X^-:= E_{\mathcal Q}^-$ and $X^+:= E^+_{\mathcal Q}$, it  suffices to prove that hypotheses $(\Phi_1)-(\Phi_6)$ in Theorem \ref{th4.6} are satisfied for $\Phi_a$.
$(\Phi_1)$ is obvious. Since $\|\cdot\|_{E_{\mathcal Q}}$ is equal to $\|\cdot\|_{H^1}$ on $E_{\mathcal Q}^+$, then by similar arguments of Lemma \ref{lm2}, $(\Phi_2)$ holds. Since $J$ is weakly sequentially lower semi-continuous in $E_{\mathcal Q}$, then by similar arguments of Lemma \ref{lm3}, $(\Phi_3)$ holds. Condition $(\Phi_4)$ holds since $J\geq0$.

The embedding ${E_a}_w^- \oplus E^+_a\hookrightarrow {E_a}_w $ is sequentially continuous. Therefore, by
Lemma \ref{J}, $J'$ is sequentially continuous on ${E_a}_w^- \oplus E^+_a$, and the same holds for $\Phi'_a$. For the same reason $J$ is sequentially lower semi-continuous on ${E_a}_w^- \oplus E^+$. Moreover $\|\cdot\|_{E_\mathcal Q}$ is sequentially lower semi-continuous on ${E_a}_w^-$. These facts together give $(\Phi_5)$.

The rest is the proof of $(\Phi_6)$.

\begin{lemma}\label{lm4.3} There is $\beta > 0$ such that for any $u \in \mathcal{K}\setminus \{0\}$ we have $\Phi_a(u) \geq\beta$.
\end{lemma}
\begin{proof}
Observe that for any $u\in E^-_a$,
$$\Phi_a(u)=-\frac{1}{2}\|u^-\|^2_{E_a}-\frac{1}{2p}\|u^-\|_{\mathcal Q^{\alpha, p}}^{2p}\leq0,$$
but for $u \in \mathcal{K}\setminus\{0\}$,
$$\Phi_a(u)=\Phi_a(u) - \frac{1}{2}\langle\Phi'_a(u),u\rangle = \left(\frac{1}{2}-\frac{1}{2p}\right)\|u\|_{\mathcal Q^{\alpha, p}}^{2p}>0$$
Therefore,
$$(\mathcal{K}\setminus\{0\})\cap E^-_a=\emptyset$$

Now, let $u\in\mathcal{K}\setminus \{0\}$.
First we show that $\|u\|_{E_\mathcal Q}$ is bounded away from $0$. By Lemma \ref{norm}, we know that the norms $\|\cdot\|_{E_\mathcal Q}$, $\|\cdot\|_{E_a}$ and $\|\cdot\|_{E_\mathcal Q}$ are equivalent on the space $E^+_a$, therefore we only need to prove $\|u\|_{E_a}=\|u^+\|_{E_a}\geq C>0$. If $\|u^+\|_{E_a}\leq1$, by $\langle\Phi'_a(u),u\rangle=0$ and \eqref{R3},
$$\|u\|^2_{E_a} = \langle J'(u),u\rangle \leq C\|u\|^{2p-1}_{E_a}\|u\|_{E_a},$$
and therefore
$$\|u\|_{E_a} \leq C\|u\|^{2p-1}_{E_a}$$
This shows that $\|u\|_{E_a}\geq C > 0$ for some independent constant $C$.

We have
$$\Phi_a(u)=\Phi_a(u)-\frac{1}{2}\langle\Phi'_a(u),u\rangle=\left(\frac{1}{2}-\frac{1}{2p}\right)J(u)$$
If $J(u)\geq1$ we have an independent positive lower bound for $\Phi_a(u)$. If $J(u)\leq1$, by (\ref{R3}) it follows
that
$$\|u\|^2_{E_a} = \langle J'(u),u\rangle \leq CJ(u)^{1-\frac{1}{2p}}\|u\|_{E_a},$$
and thus
$$\|u\|_{E_a}\leq CJ(u)^{1-\frac{1}{2p}}.$$
Therefore
$$\Phi_a(u) \geq C > 0$$ for some
independent $C$ since $\|u\|_{E_a}$ is bounded away from 0 on $\mathcal{K}\setminus\{0\}$ as shown above.
\end{proof}

\begin{lemma}\label{lm4.3}
 The $(PS)_{c}$ sequence $\{u_n\}$ satisfying
 \begin{equation*}\label{e2.3}
\Phi_a(u_n) \to c,\quad \|\Phi'_a(u_n)\|_{(E_{\mathcal Q})^*}\to0,\quad \text{as}~n\to\infty.
\end{equation*}
 is bounded in $E_{\mathcal Q}$.
\end{lemma}
\begin{proof}
The proof is similar to Lemma \ref{lm4.1}, we omit it.
\end{proof}
In the following lemma $\beta$ denotes the constant given by Lemma \ref{lm4.3}.
\begin{lemma}\label{lm2.4}
For $c \in \R$ let $\{u_n\}\subset E_{\mathcal Q}$ be a $(PS)_{c}-$sequence for $\Phi_a$. Then either $c= 0$ and $u_n \to 0$ or $c\geq \beta$ and there are $k\in \mathbb{N}$, $k \leq [c/\beta]$,
and for each $1 \leq i \leq k$ a sequence $\{k_{i,n}\}_n\subset\mathbb{Z}^N$ and a function $v_i \in E_{HL} \setminus\{0\}$ such
that, after extraction of a subsequence of $\{u_n\}$,
$$\left\|u_n-\sum_{i=1}^{k}\tau_{k_{i,n}}v_i\right\|_{E_\mathcal Q}\to0,$$
$$\Phi_a\left(\sum_{i=1}^{k}\tau_{k_{i,n}}v_i\right)\to\sum_{i=1}^{k}\Phi_a\left(v_i\right)=c,$$
$$|k_{i,n}-k_{j,n}|\to\infty\quad\text{for}~i\neq j,$$
$$\Phi'_a(v_i)=0\quad\text{for~all}~i.$$
\end{lemma}
\begin{proof}
The proof follows the same lines as for the proof of Lemma 4.5 in \cite{A}, so we omit it.
\end{proof}

Given any compact interval $I\subset(0,\infty)$ with $d = \max I$ we set $k = [d/\beta]$ and
$$[\mathcal{F},k]=\left\{\sum_{i=1}^j\tau_{m_i}v_i : 1\leq j\leq k, m_i\in\mathbb{Z}^N, v_i\in\mathcal{F} \right\}$$
By Lemma \ref{lm2.4}, $[\mathcal{F}, k]$ is a $(PS)_I-$attractor.

Since the projections $P_{a,0}, Id-P_{a,0}$ commute with
the action of $\mathbb{Z}^N$ on $E_{HL}$, it is easy to get $(\Phi_6)$, see \cite[Prop. 2.57]{VR}. Therefore, by Theorem \ref{th4.6}, we get infinitely many
geometrically distinct solutions. This concludes the proof of Theorem \ref{th3}.

\vskip4mm
{\section{Final remarks}}
 \setcounter{equation}{0}
In this section, we discuss the case $\lambda\to a+$. What may happen if $a$ is in the spectrum of $-\Delta+V$, or, equivalently, if $0$ is in the spectrum of $-\Delta+V-a$? It is commonly believed that in this case well localized solutions of the corresponding local Schr\"odinger equation do not exist (see \cite{Pankov2}, section 7). Indeed, if $u$ solves the local equation
$$-\Delta u+V(x)u=f(x,u)+a\cdot u
$$
 then, by setting $W(x):=-f(x,u)/u$, $u$ is also a solution of
$$
-\Delta u +(V(x)+W(x))u=a\cdot u 
$$
If $W(x)\to 0$ as $|x|\to \infty$, we are in front of a periodic Schr\"odinger operator perturbed by a decaying potential.  If $W(x)$ decays sufficiently fast, it defines a relatively compact perturbation of the Schr\"odinger operator, so that the essential spectrum does not change (see \cite{K}, Theorem 5.35), and $a$ turns out to be an eigenvalue lying in the essential spectrum. So, the problem of the non existence of well localized solutions (at least $H^1$) is strictly related to the non existence of embedded eigenvalues for periodic Schr\"odinger operators perturbed  by decaying potentials. This problem seems far from being solved in its generality.  Some interesting results have been proved for the  one-dimensional case (Hill's equation). In this case, it is known that the endpoints of the continuous spectrum for Hill'sequation  can be characterized by special eigenvalues of associated eigenvalue problems
with periodic, respectively anti-periodic boundary conditions. In particular it was shown that the branches of solutions (bifurcating both from zero of from infinity) exist globally, and only over the gaps they consist of square-integrable solutions. Hence they "disappear" within the $L^2$-setting when reaching the spectrum and reoccur when reaching the gap again (see \cite{KM, AY, KS}).
\par \medskip

The same question arises in the nonlocal case. Furthermore, what about the behaviour of the branches of solutions $u_\lambda$, as $\lambda\to a^+$? We can prove that if there exists no $H^1$-solution for \eqref{Choqeq} in the left borderline point of the spectrum, $\lambda =a$, then the branches of solutions bifurcate from $\infty$ in $a$:
\begin{proposition}\label{th4} Let $N\geq3$, $\alpha\in(0,N)$, $p\in(\frac{N+\alpha}{N},\frac{N+\alpha}{N-2})$ and $(V_1)-(V_2)$ hold. 	Let us assume that problem
\begin{equation}\label{1}
\left\{
\begin{array}{ll}
\aligned
&-\Delta u+V(x)u=\left(I_\alpha\ast |u|^{p}\right)|u|^{p-2}u+\lambda u \quad \text{in}~\R^N, \\
&\quad u \in H^{1}\left(\mathbb{R}^{N}\right)
\endaligned
\end{array}
\right.
\end{equation}
has for $\lambda=a$ only the trivial solution. For any $\lambda\in(a,b)$, let $u_\lambda$ be a solution obtained in Theorem \ref{th1}, then
\begin{equation}\label{5}
\|u_\lambda\|_{H^1}\to+\infty~~\text{as}~~\lambda\to a^+.
\end{equation}
Moreover, $a$ is the only possible gap-bifurcation (from infinity) point for \eqref{Choqeq} in $[a,b]$.
\end{proposition}
\begin{proof}
We prove (\ref{5}) by contradiction. Suppose that $$\limsup\limits_{\lambda\to a^+}\|u_\lambda\|_{H^1}<+\infty.$$
Take a sequence $\{\lambda_n\}\subset[a,b]$ such that $\lim\limits_{n\to+\infty}\lambda_n=a$. Then $\{u_{\lambda_n}\}$ is bounded in $H^1(\R^N)$ and
$$\Phi_{\lambda_n}(u_{\lambda_n})=c_{\lambda_n},\quad \Phi'_{\lambda_n}(u_{\lambda_n})=0.$$
Further, by \eqref{ka2} we know that for $\lambda_n<0$ 
\begin{equation}\label{5}
c_{\lambda_n}\geq\kappa(\lambda_n)\geq\left(\frac{1}{8}-\frac{1}{2^{2p+1}}\right)
\left(\frac{\beta_{\lambda_n}^p}{2C(N,\alpha,p)}\right)^{\frac{1}{p-1}}=\left(\frac{1}{8}-\frac{1}{2^{2p+1}}\right)
\left(\frac{\beta_{0}^p}{2C(N,\alpha,p)}\right)^{\frac{1}{p-1}}>0.
\end{equation}
Take a subsequence of $\{u_{\lambda_n}\}$ so that $u_{\lambda_n}\rightharpoonup u$ as $n\to+\infty$.
Since $\Phi'_a$ is weakly sequentially continuous, for any $v\in H^1(\R^N)$ we have
$$\langle\Phi'_a(u),v\rangle=\langle\Phi'_a(u_{\lambda_n}),v\rangle+o(1)=\langle\Phi'_{\lambda_n}(u_{\lambda_n}),v\rangle+o(1).$$
Thus $u$ is a weak solution of \eqref{1}. We assert that $u\not\equiv0$. Indeed, if for any $r>0$ 
\begin{equation}\label{e6.6}
\lim_{n\to\infty}\sup_{y\in \R^N}\int_{B_r(y)}u_{\lambda_n}^2dx=0,
\end{equation}
by Lions' concentration compactness principle \cite[Lemma 1.21]{Willem},
$$u_{\lambda_n}\to 0~\text{in}~ L^q(\R^N)~\forall  q\in(2, 2^*).$$
Then, by the Hardy-Littlewood-Sobolev inequality, we have
\begin{equation*}
J(u_{\lambda_n})\leq C|u_{\lambda_n}|^{2p}_{\frac{2Np}{N+\alpha}}\to0,
\end{equation*}
which combined with
$$c_{\lambda_n}=\Phi_{\lambda_n}(u_{\lambda_n})-\frac{1}{2}\langle\Phi'_{\lambda_n}(u_{\lambda_n}),u_{\lambda_n}\rangle
=\left(\frac{1}{2}-\frac{1}{2p}\right)J(u_n)$$
implies $c_{\lambda_n}\to 0$. This contradicts with \eqref{5}.

Therefore \eqref{e6.6} does not hold, so that there is an $r>0$ and a sequence $\{y_n\}$ such that 
$$
\lim_{n\to\infty}\int_{B_r(y_n)}u_{\lambda_n}^2dx=\delta>0
$$
We can assume, without loss of generality, that for any $n$
$$
\int_{B_r(y_n)}u_{\lambda_n}^2dx\geq \frac \delta 2
$$
 Now we choose $k_{n} \in \mathbb{Z}^N$ such that $|k_{n} - y_n | = \min \{|k- y_n | : k \in \mathbb{Z}^N\} $: note that $|k_{n} - y_n |\leq \sqrt N/2$. Set $v_{n}:= \tau_{k_{n}}u_{\lambda_n}=u_{\lambda_n}(\cdot+k_{n})$. Hence we have
\begin{equation}\label{e6.8}
\int_{B_{1+\sqrt{N}/2}(0)}v_{n}^2dx\geq\frac{\delta}{2}.
\end{equation}
Moreover, $\|v_n\|_{H^1} =\|u_{\lambda_n}\|_{H^1}$ is also bounded in $H^1(\R^N)$. Thus, up to a subsequence, $v_{n}\rightharpoonup v$ weakly in $H^1(\R^N)$ and $v_n(x) \to v(x)$ almost everywhere in $\R^N$. By Sobolev compact embedding, $v_{n}\to v$ strongly in $L^2_{loc}(\R^N)$, where $v\not\equiv 0$ thanks to \eqref{e6.8}. By the invariance of $\Phi_a$ under the action of $\mathbb{Z}^N$, we have $$\Phi'_a(v)=\Phi'_a(u)=0.$$
Thus we get a nontrivial weak solution $v$ for \eqref{Choqeq}, which  contradicts our assumption.

Moreover, since $\lim\limits_{\lambda\to d}\|u_\lambda\|_{H^1}<\infty$ for any $d\in(a,b)$ by Theorem \ref{th1}, $a$ is the only possible gap-bifurcation (from infinity) point for \eqref{Choqeq} in $[a,b]$. This completes the proof.
\end{proof}

\begin{remark}
It is easy to observe that, when $V \equiv 0$, problem \eqref{1} has for $\lambda\geq 0$ only the trivial solution. Indeed, if $V \equiv 0$ then $\sigma(-\Delta + V ) = [0, +\infty)$. If $u$ is a solution of \eqref{1}, by testing the equation against $u$, we obtain the identity
\begin{equation}\label{2}
\int|\nabla u|^2dx-\lambda\int|u|^2dx-\int(I_\alpha\ast |u|^{p})|u|^pdx=0.
\end{equation}
Moreover, we have the Poho\v{z}aev identity \cite[Theorem 3]{MS3}
\begin{equation}\label{3}
\frac{N-2}{2}\int|\nabla u|^2dx-\frac{N}{2}\lambda\int|u|^2dx-\frac{N+\alpha}{2p}\int(I_\alpha\ast |u|^{p})|u|^pdx=0.
\end{equation}
Combining \eqref{2} and \eqref{3}, we have
\begin{equation}\label{4}
\left(\frac{N-2}{2}-\frac{N+\alpha}{2p}\right)\int|\nabla u|^2dx=\lambda\left(\frac{N}{2}-\frac{N+\alpha}{2p}\right)\int|u|^2dx.
\end{equation}
Since $\frac{N+\alpha}{N}<p<\frac{N+\alpha}{N-2}$, 
$$\frac{N-2}{2}-\frac{N+\alpha}{2p}<0\quad \text{and}\quad
\frac{N}{2}-\frac{N+\alpha}{2p}>0.$$
Then if $\lambda\geq0$ we have from \eqref{4} that $u\equiv0.$
\end{remark}
\vskip4mm
\noindent{\bf Acknowledgments}\\
We are  grateful to C.A. Stuart for pointing out several interesting ideas related to open problems presented in the last Section.\\
The first  author is supported by the National Natural Science Foundation of China(No. 11901532). The third author is partially supported by Gnampa-Indam.

\vskip4mm

 {}


\begin{thebibliography}{99}

\bibitem{A} 
\newblock N. Ackermann, 
\newblock \textit{On a periodic Schr\"{o}dinger equation with nonlocal superlinear part}
\newblock, Math. Z. \textbf{248} (2004) 423--443.

\bibitem{AY}
\newblock S. Alama and Y.Y. Li,
\newblock \emph{Existence of solutions for semilinear elliptic equations with indefinite linear part},
\newblock J. Diff. Equations \textbf{96} (1992), 89--115.

\bibitem{Al}
\newblock C. Albanese,
\newblock \emph{Localised solutions of Hartree equations for narrow-band crystals},
\newblock Comm. Math. Phys. \textbf{120} (1988), 97--103.

\bibitem {ACM}
\newblock C. Alves, M. Cavalcante and E. Medeiros, 
\newblock \emph{A semilinear Schr\"odinger equation with zero on the boundary of the spectrum and exponential growth in $\R^2$},
\newblock Commun. Contemp. Math. \textbf{21} (2019).

\bibitem{AM}
\newblock A. Ambrosetti and  A. Malchiodi,
\newblock \emph{Perturbation Methods and Semilinear Elliptic Problems on $\R^n$}, \newblock Progress in Mathematics, Vol. 240. Birkh\"{a}user, Basel, 2006.

\bibitem{BD} 
\newblock T. Bartsch and Y. H. Ding, 
\newblock \textit{On a nonlinear Schr\"{o}dinger equation with periodic potential},
\newblock Math. Ann. \textbf{313} (1999), 15--37.


\bibitem{B} 
\newblock A. S. Besicovitch, 
\newblock \textit{Almost Periodic Functions},
\newblock Cambridge Univ. Press, Cambridge, 1932.



\bibitem{BJS}
\newblock B. Buffoni, L. Jeanjean and C.A. Stuart, 
\newblock \emph{Existence of a nontrivial solution to a strongly indefinite semilinear equation},
\newblock Proc. Am. Math. Soc. \textbf{119} (1993) 179--186.

\bibitem{CVZ}
\newblock D.~Cassani, J.~Van Schaftingen and J.~Zhang, 
\newblock \textit{Groundstates for Choquard type equations with Hardy--Littlewood--Sobolev lower critical exponent}, 
\newblock Proc. Roy. Soc. Edinburgh Sect. A \textbf{150} (2020), 1377--1400.

\bibitem{CZ}
\newblock D.~Cassani and J.~Zhang, 
\newblock \textit{Choquard-type equations with Hardy--Littlewood--Sobolev upper-critical growth},
\newblock Adv. Nonlinear Anal. \textbf{8} (2019), 1184--1212.


\bibitem{CZR}
\newblock V. Coti-Zelati and P.H. Rabinowitz,
\newblock \emph{Homoclinic type solutions for a semilinear elliptic PDE on $\R^N$}, 
\newblock Comm. Pure Appl. Math. \textbf{45} (1992), 1217--1269.

\bibitem{D}
\newblock Y. H. Ding,
\newblock \textit{Variational Methods for Strongly Indefinite Problems}, \newblock Interdisciplinary Mathematical Sciences, Vol. 7, World Scientific Publisher, 2007.

\bibitem{Du}
\newblock L.~Du and M.~Yang, 
\newblock \textit{Uniqueness and nondegeneracy of solutions for a critical nonlocal equation},
\newblock  Discrete Contin. Dyn. Syst. \textbf{39} (2019), 5847--5866.


\bibitem{E} 
\newblock M. S. P. Eastham,
\newblock \textit{The spectral theory of periodic differential equations}, \newblock Scottish Academic Press Ltd., Edinburgh and London, 1973.


\bibitem{fro}
\newblock H.~Fr\"ohlich, 
\newblock \textit{Theory of electrical breakdown in ionic crystal}, 
\newblock Proc. Roy. Soc. Ser. A \textbf{160} (1937), 230--241.


\bibitem{GT} 
\newblock D. Gilbarg and N. Trudinger, 
\newblock \textit{Elliptic partial differential equations of second order}, \newblock Springer, Berlin Heidelberg New York Tokyo, 1983.


\bibitem{HKS} 
\newblock H.-P. Heinz, T. K\"{u}pper and C. A. Stuart, 
\newblock \textit{Existence and bifurcation of solutions for nonlinear perturbations of the periodic Schr\"{o}dinger equation},
\newblock J. Differential Equations \textbf{100} (1992), 341--354.

\bibitem{K} 
\newblock T. Kato,
\newblock \textit{Perturbation theory for linear operators. Second edition}, \newblock Springer-Verlag, Berlin-New York, 1976.

\bibitem{KM}
\newblock T. K\"upper and T. Mrziglod,
\newblock \emph{On the bifurcation structure of nonlinear perturbations of Hill's equations at boundary points of the continuous spectrum},
\newblock SIAM J. Math. Anal. \textbf{26} (1995), 1284--1305.

\bibitem{KS}
\newblock T. K\"upper,and C. Stuart, 
\newblock \emph{Bifurcation into gaps in the essential spectrum},
\newblock J. Reine Angew. Math. \textbf{409} (1990), 1--34.

 \bibitem{KS1}
 \newblock T. K\"upper,and C. Stuart, 
 \newblock \emph{ Necessary and sufficient conditions for gap-bifurcation}, \newblock Nonlinear Anal. \textbf{18} (1992),  893--903.
 
 \bibitem{LL}
 \newblock E.H.~Lieb and M.~Loss \newblock \emph{Analysis. Second edition}, \newblock Graduate Studies in Mathematics 14. American Mathematical Society, Providence, RI, 2001.

\bibitem{Lieb}
\newblock E. H. Lieb,
\newblock \emph{Existence and uniqueness of the minimizing solution of Choquard's nonlinear equation,}
\newblock  Studies in Appl. Math. \textbf{57} (1976/77),  93--105.

\bibitem{Lions} 
\newblock P.-L. Lions, 
\newblock \emph{The Choquard equation and related questions,}
\newblock   Nonlinear Anal. 4 (1980) 1063--1072.

\bibitem{M}
\newblock J. Mederski, 
\newblock \emph{Solutions to a nonlinear Schr\"odinger equation with periodic potential and zero on the boundary of the spectrum},
\newblock Topol. Methods Nonlinear Anal. \textbf{46} (2015),  755--771.

\bibitem{MMVS}
\newblock C. Mercuri, V. Moroz, and J. Van Schaftingen,
\newblock \textit{Groundstates and radial solutions to nonlinear Schrödinger-Poisson-Slater equations at the critical frequency},
\newblock Calc. Var. Partial Differential Equations \textbf{55} (2016).

\bibitem{MS1} 
\newblock V. Moroz and J. Van Schaftingen, 
\newblock \textit{Ground states of nonlinear Choquard equations: Existence, qualitative properties and decay asymptotics},
\newblock J. Funct. Anal. \textbf{265} (2013) 153--184.

\bibitem{MS2}
\newblock V.~Moroz and J.~Van Schaftingen, 
\newblock \textit{A guide to the Choquard equation},
\newblock J. Fixed Point Theory Appl. \textbf{19} (2017), 773--813.

\bibitem{MS3}
\newblock V. Moroz and J. Van Schaftingen,
\newblock \textit{Existence of groundstates for a class of nonlinear Choquard equations},
\newblock Trans. Amer. Math. Soc. \textbf{367} (2015), 6557--6579.



\bibitem{Pankov1}
\newblock A. Pankov, 
\newblock{Semilinear elliptic equations in $\R^n$ with nonstabilizing coefficients}, 
\newblock (Russian) Ukrain. Mat. Zh. \textbf{41} (1989), 1247--1251; translation in Ukrainian Math. J. \textbf{41} (1989), 1075–1078.

\bibitem{Pankov2} 
\newblock A. Pankov, 
\newblock \textit{Periodic nonlinear Schr\"{o}dinger equation with application to photonic crystals},
\newblock Milan J. Math. \textbf{73} (2005), 259--287.

\bibitem{P1}
\newblock S. Pekar,
\newblock \emph{Untersuchung \"{u}ber die Elektronentheorie der Kristalle},
\newblock Akademie Verlag, Berlin, 1954.



\bibitem{QRT} 
\newblock D. Qin, V. D. R\u{a}dulescu and X. Tang, 
\newblock \textit{Ground states and geometrically distinct solutions for
periodic Choquard-Pekar equations},
\newblock Journal of Differential Equations \textbf{275} (2021), 652--683.

\bibitem{RS}
\newblock M. Reed, B. Simon, 
\newblock \emph{Methods of Modern Mathematical Physics, IV, Analysis of Operators},
\newblock Academic Press, 1978.


\bibitem{Ru}
\newblock D. Ruiz,
\newblock \emph{On the Schrödinger-Poisson-Slater system: behavior of minimizers, radial and nonradial cases},
\newblock Arch. Ration. Mech. Anal. \textbf{198} (2010).

\bibitem{Sc}
\newblock M. Schechter,
\newblock \emph{Nonlinear Schr\"odinger operators with zero in the spectrum},
\newblock Z. Angew. Math. Phys. \textbf{66} (2015), 2125--2141.


\bibitem{VR} 
\newblock V. C. Z. Sissa and P. H. Rabinowitz,
\newblock \textit{Homoclinic type solutions for a semilinear elliptic PDE on $\R^N$},
\newblock Comm. Pure Appl. Math. \textbf{45} (1992),  1217--1269.

\bibitem{S1} 
\newblock C.A. Stuart, 
\newblock \emph{Bifurcation into spectral gaps},
\newblock Supplement to the Belgian Mathematical Society, 1995.

\bibitem{S2}
\newblock C.A. Stuart, 
\newblock \emph{Bifurcation from the essential spectrum},
\newblock Topological nonlinear analysis, II (Frascati, 1995), 397--443, Progr. Nonlinear Differential Equations Appl., 27, Birkh\"auser Boston, Boston, MA, 1997.

\bibitem{S3}
\newblock C.A. Stuart, 
\newblock \emph{Bifurcation from the continuous spectrum in the $L^2$-theory of elliptic equations on $\R^n$},
\newblock  Recent methods in nonlinear analysis and applications (Naples, 1980), 231--300, Liguori, Naples, 1981.



\bibitem{T} 
\newblock C. Troestler,
\newblock \textit{Bifurcation into spectral gaps for a noncompact semilinear Schr\"{o}dinger equation with nonconvex potential},
\newblock arXiv preprint, arXiv:1207.1052 (2012).

\bibitem{TW}
\newblock C. Troestler and M. Willem, 
\newblock \emph{Nontrivial solution of a semilinear Schr\"odinger equation},
\newblock Comm. Partial Differential Equations \textbf{21} (1996), 1431--1449.

\bibitem{YCD}
\newblock M. Yang, W. Chen and Y. Ding,
\newblock \emph{Solutions for periodic Schr\"odinger equation with spectrum
zero and general superlinear nonlinearities},
\newblock J. Math. Anal. Appl. \textbf{364} (2010), 404--413.
	
	\bibitem{W}
	\newblock T.~Wang and T.~Yi, 
	\newblock \textit{Symmetry of separable functions and its applications to Choquard type equations},
	\newblock Calc. Var. Partial Differential Equations \textbf{59} (2020), 23 pp.
	
\bibitem{Willem} 
\newblock M. Willem,
\newblock \textit{Minimax theorems},
\newblock Progress in Nonlinear Differential Equations and their Applications, 24, Boston, MA: Birkh\"{a}user Boston Inc., 1996.

\bibitem{WZ}
\newblock M. Willem and W. Zou,
\newblock \emph{On a Schr\"odinger equation with periodic potential and spectrum point zero},
\newblock Indiana Univ. Math. J. \textbf{52} (2003), 109--132.

 \end{thebibliography}
\end{document}